\documentclass{article}

\usepackage[utf8]{inputenc} 
\usepackage[T1]{fontenc}    
\usepackage{float}
\usepackage{url}            
\usepackage{booktabs}       
\usepackage{cellspace}
\usepackage{amsfonts}       
\usepackage{nicefrac}       
\usepackage{microtype}      
\usepackage{xcolor}         
\usepackage{breqn}

\usepackage{hyperref}       
\usepackage{biblatex}
\usepackage{color}

\usepackage{comment}
\usepackage{amsmath}
\usepackage{amsthm}
\usepackage{amssymb}
\usepackage{algorithm}
\usepackage{algpseudocode}
\usepackage{graphicx}
\usepackage{xfrac}
\usepackage{multirow}
\usepackage{pifont}

\newtheorem{theorem}{Theorem}[section]
\newtheorem{corollary}[theorem]{Corollary}
\newtheorem{lemma}[theorem]{Lemma}
\newtheorem{assumption}{Assumption}
\newtheorem{fact}[theorem]{Fact}

\theoremstyle{definition}
\newtheorem{definition}{Definition}[section]
\newtheorem{example}{Example}[section]

\makeatletter
\newcommand{\newreptheorem}[2]{%
\newenvironment{rep#1}[1]{%
 \def\rep@title{#2~\ref{##1}}%
 \begin{rep@theorem}}%
 {\end{rep@theorem}}}
\makeatother

\newreptheorem{theorem}{Theorem}
\newreptheorem{lemma}{Lemma}

\newcommand{\norm}[1]{\left\lVert#1\right\rVert}

\let\top\intercal

\title{Accelerated Mirror Descent for Non-Euclidean Star-convex Functions}

\begin{document}

\author{Clement Lezane, Sophie Langer, Wouter M. Koolen}
\maketitle

\begin{abstract}
Acceleration for non-convex functions is a fundamental challenge in optimisation. We revisit star-convex functions, which are
strictly unimodal on all lines through a minimizer.
\cite{pmlr-v125-hinder20a} accelerate unconstrained star-convex minimization of functions that are smooth with respect to the Euclidean norm. To do so, they add a certain binary search step to gradient descent.
In this paper, we accelerate unconstrained star-convex minimization of functions that are \textit{weakly} smooth with respect to an \textit{arbitrary} norm. We add a binary search step to mirror descent, generalize the approach and refine its complexity analysis.
We prove that our algorithms have sharp convergence rates for star-convex functions with $\alpha$-H\"older continuous gradients and demonstrate that our rates are nearly optimal for $p$-norms.
\end{abstract}


\section{Introduction}

Accelerated gradient descent by~\cite{Nesterov1983AMF} has proven to be a powerful tool for first-order optimization, significantly improving algorithmic performance. Nowadays several extensions of accelerated methods exist including coordinate descent methods \cite{coord1, coord2}, distributed gradient descent \cite{Quli16} and proximal methods \cite{prox1} with applications across widespread domains such as image deblurring \cite{BT09}, compressed sensing \cite{Becker11} and deep learning \cite{acc-ex2}.

Theoretical results analysing accelerated methods mainly focus on smooth convex optimisation, demonstrating optimal convergence and improved performance against traditional gradient methods, see~\cite{Nesterov1983AMF}, ~\cite{bubeck2014convex} and \cite{AGJ18} for further references. However, these assumptions are violated in modern machine learning applications due to their non-convex energy landscapes. Attempts to understand acceleration for non-convex functions are made, for instance, by leveraging the Łojasiewicz inequalities \cite{bolte10}. Further classes of non-convex functions are discussed by~\cite{Carmon2017ConvexUP} and~\cite{pmlr-v125-hinder20a}.

\cite{Nesterov2006CubicRO} relax the convexity assumption to so-called star-convex functions, a class of structured non-convex functions that are strictly unimodal on all lines through a minimizer.  While convex functions exhibit unimodality along \textit{all} lines, star-convexity is a weaker condition. \cite{Lee16} provide examples of functions that are star-convex but not convex. Additionally, \cite{Zhou2019SGDCT} and \cite{pmlr-v80-kleinberg18a} present evidence that neural network loss functions often are star-convex (but not convex), emphasizing its relevance for modern machine learning applications.

For star-convex functions,  \cite{pmlr-v125-hinder20a} present an accelerated framework using binary search techniques. Their method is nearly optimal for objective functions having a Lipschitz gradient with respect to the Euclidean norm. However, recent machine learning results such as by~\cite{Yang2015, Adil19} minimize the objective functions in \textit{non}-Euclidean norms, which are not covered by the findings of~\cite{pmlr-v125-hinder20a}. This gap inspires us to extend and generalize their algorithm to arbitrary normed spaces. We combine a generalized ``line-search'' style algorithm \cite{pmlr-v125-hinder20a} with accelerated mirror descent to minimize star-convex functions with $\alpha$-H\"older continuous gradients  (also known as weakly smooth functions, see Definition~\ref{def:weak smooth}). We demonstrate that the general non-Euclidean acceleration framework is not related to the convexity of the objective function, as previously theorized. Rather, a key factor appears to be the inherent regularity characteristic of norms (Fact~\ref{def-regular-norm}). Our finding challenges the common belief that the properties of the objective function are more crucial than those of the norm.

We demonstrate that in case of $p$-norms, our algorithm is nearly optimal for $p>1$. Our approach further extends to the $L_1$ norm though with an additional dimensionality-dependent cost tied to challenges in constructing an appropriate so-called distance generating function. By showing that non-convex functions can be minimized using accelerated algorithms in non-Euclidean geometry, we broaden the conventional understanding of which functions admit acceleration. Our findings suggest that a wider class of non-convex functions can be efficiently minimized, potentially driving further research in this direction.

\textbf{Notation.} Throughout this paper, we use $\|\cdot\|$ to denote an arbitrary norm and write $\|\cdot\|_\ast$ for the associated dual norm. We denote a minimizer of a function $F:\mathbb{R}^d \to \mathbb{R}$ by $x_\star$ and we say that a point is $\epsilon$-optimal if $F(x) \leq F(x_\star)+\epsilon$. We denote the scalar product  by $\langle x,y\rangle=\sum_{i=1}^d x_i y_i$. We say that for two functions $f$ and $g$, $f(x)=\mathcal{O}(g(x))$ if there exist positive real numbers $M$ and $x_0$ such that $|f(x)|\leq Mg(x)$ for all $x \geq x_0$. Further, $f(x)=\Omega(g(x))$ if $\lim\sup_{x \to \infty} \Big|\frac{f(x)}{g(x)}\Big|>0$.

\subsection{Literature review} 

Previous results on star-convex functions primarily focus on non-accelerated algorithms. \cite{ Gower2020SGDFS} and \cite{Hardt2016GradientDL} show convergence guarantees for stochastic gradient descent, \cite{Nesterov2006CubicRO} analyze cubic regularization and \cite{Lee16} investigate a cutting plane method. \cite{joulani17} apply the star-convexity assumption to stochastic and online learning settings, showing regret bounds for first-order methods, leading to a $\mathcal{O}(1/T)$ convergence rate (with $T$ being the number of gradient steps) in case that either the objective function or its gradient are Lipschitz.

Acceleration on star-convex functions has been studied by, e.g.,\ \cite{Guminov2018AcceleratedMF} for low-dimensional subspace optimization and by~\cite{Nest19} for an  one-dimensional line search-type of algorithm, both showing convergence rates to an $\epsilon$-optimal point. However, as noted by~\cite{pmlr-v125-hinder20a}, these algorithms can be limited by their potentially high computational costs for minimising general star-convex functions.

\cite{pmlr-v125-hinder20a} propose an accelerated framework using binary search, showing near-optimal performance for $\tau$-star-convex functions (see Definition~\ref{def star-convex}) that are $L$-smooth (see Definition~\ref{def:weak smooth}) with respect to the Euclidean norm in a domain bounded by radius $R$. Similar to~\cite{Guminov2018AcceleratedMF}, their algorithm finds an $\epsilon$-optimal solution in $\mathcal{O}(\epsilon^{-1/2}L^{1/2} \tau R \log(LR^2/\epsilon))$ iterations. For smooth star-convex functions, they also provide a lower bound  of $\Omega(L^{1/2} \tau R \epsilon^{-1/2})$.

While the accelerated framework of~\cite{pmlr-v125-hinder20a}  is limited to star-convex functions that are smooth with respect to the Euclidean norm, Nesterov's accelerated framework \cite{Elster1993ModernMM, NEMIROVSKII198521} for convex functions applies to general normed spaces under a weaker smoothness assumption (Definition~\ref{def:weak smooth} below).

\subsection{Our contributions}
In this paper, we generalize the accelerated framework introduced by~\cite{pmlr-v125-hinder20a} to encompass weakly smooth star-convex functions with respect to a general norm. To address the challenges arising from the non-convexity of the objective function and the non-Euclidean geometry of the norm, we present a new convergence proof by using a regularity property of norms, involving novel arguments for bounded iterates to quantify convergence. Our key contributions are summarized as follows:

\textbf{Acceleration for arbitrary norms.}
To extend the acceleration framework of~\cite{pmlr-v125-hinder20a} beyond Euclidean norms to general norms, we examine regularity properties of the derivative of the squared norm $\|\cdot\|^2$. We show that star-convexity and weak smoothness suffice to apply our generalized accelerated framework. Our results underscore that the key to acceleration lies in the convexity and homogeneity of norms, rather than convexity of the objective function.

\textbf{Extended binary search.} We generalize the binary search algorithm introduced by~\cite{pmlr-v125-hinder20a}, originally designed for functions that are smooth with respect to the Euclidean norms, to encompass objective functions that are weakly smooth  under arbitrary norms. We provide a streamlined convergence proof (Theorem~\ref{thm-binary-search}) for our generalized algorithm and show that the number of oracle calls required for our binary search step is at most logarithmic. A key challenge in this setting is that our iterates may be unbounded, unlike in the Euclidean case. Additionally, we provide a lower bound for the binary search step when applied to smooth  non-star-convex functions. 

\textbf{Nearly optimal performance for $p$-norm.} For $p$-norms, we prove the same convergence rates as shown for convex functions by~\cite{NEMIROVSKII198521}, extending them to star-convex functions. According to the lower bounds by~\cite{GUZMAN20151}, our rates are nearly optimal, up to a logarithmic factor. Table~\ref{our:results} contextualizes our results in the existing literature.

\begin{table}[ht]
\centering
\begin{tabular}{|Sc|Sc|Sc|Sc|}
  \hline
  $p$
  & Bound
  & Convex
  &  $\tau$-star-convex
  \\ [0.5ex]
  \hline\hline
  \multirow{4}{*}{$(1,2]$}
  & Upper & $\mathcal{O}_{\kappa} \left(\big(\frac{L  R^\kappa}{\epsilon}\big)^{\frac{2}{3 \kappa-2}}  \right)$  & $\mathcal{O}_{\kappa} \left(\big(\frac{L \tau^{2} R^\kappa}{\epsilon}\big)^{\frac{2}{3\kappa - 2 }} \log^2 \Big( \frac{L \tau R}{\epsilon} \Big) \right) $\\
  & & \cite{NEMIROVSKII198521} & \textbf{ Our work} (Corollary~\ref{cor-p<2}(i))\\
  \cline{2-4}
  &Lower & \multicolumn{2}{|Sc|}{$\Omega_{\kappa} \left(\big(\frac{L R^\kappa}{\epsilon[\ln d]^{\kappa-1}}\big)^{\frac{2}{3 \kappa-2}}\right)$ }\\
  & & \multicolumn{2}{|Sc|}{\cite{GUZMAN20151}}
  \\
  \hline \hline
  \multirow{4}{*}{$[2,\infty]$}
&Upper  & $\mathcal{O}_{p,\kappa} \left(\big(\frac{L R^\kappa}{\epsilon}\big)^{\frac{p}{\kappa p+\kappa-p}}\right) $  & $\mathcal{O}_{p,\kappa} \left(\big(\frac{L \tau^{p} R^\kappa}{\epsilon}\big)^{\frac{p}{\kappa p+\kappa-p}} \log^2 \Big( \frac{L \tau R}{\epsilon} \Big) \right)$ \\
& & \cite{NEMIROVSKII198521} & \textbf{Our work} (Corollary~\ref{cor-p<2}(ii))\\
\cline{2-4}
  &Lower & \multicolumn{2}{|Sc|}{$\Omega_{p,\kappa} \left(\big(\frac{L R^\kappa}{\epsilon}\big)^{\frac{p}{\kappa p+\kappa-p}} \right)$}
  \\
  & & \multicolumn{2}{|Sc|}{\cite{GUZMAN20151}}\\
 \hline
\end{tabular}
\caption{Summary of the relation between existing work and our results for $\tau$-star-convex functions. $R$ is the radius of the domain of the $\tau$-star-convex objective function, $(L, \kappa)$ is its smoothness with respect to a $p$-norm $\| \cdot \|_p$  with $p>1$ and $\epsilon$ the accuracy of the algorithm. Since convex functions are $1$-star-convex, the lower bounds from~\cite{GUZMAN20151} apply to both. 
} \label{our:results}
\end{table}

\bigskip\noindent
The article is organized as follows. In Section~\ref{sec2}, we outline the problem setting and provide preliminaries. Section~\ref{sec3} presents our generalized accelerated algorithm and discusses its convergence guarantees for various settings. Section~\ref{sec4} introduces our binary search sub-algorithm and its running time. In Section~\ref{sec-application} we put the two ingredients together and obtain our main bound. A discussion of future research directions is given in 
Section~\ref{sec5}. All proofs are provided in the Appendix.

\section{Problem setting and preliminaries}
\label{sec2}
In this paper we consider the unconstrained minimization problem 
\begin{align*}
    \min_{x \in \mathbb{R}^d} F(x)
\end{align*}
with the unknown function $F \in \mathcal{F}$ coming from a given class $\mathcal{F} \subset [\mathbb{R}^d \to \mathbb R]$ of differentiable functions. We briefly discuss constrained domains in Section~\ref{sec5}. We assume a first-order access model, meaning that the algorithm has access to the objective function as well as its gradient. The performance of the algorithm is measured by comparing (after $T$ iterations) the function value of its output against $\min_{x \in \mathbb{R}^d} F(x)$. In the following, we assume that $\mathcal{F}$ describes the class of star-convex, (weakly) smooth functions with respect to a given (not necessarly Euclidean) norm $\|\cdot\|$. The formal definitions for this function class are introduced next.

\subsection{Definitions}

We consider the class of $\tau$-star-convex functions, as introduced by~\cite{Hardt2016GradientDL}, which are defined as follows. 
\begin{definition}[$\tau$-star-convexity] 
\label{def star-convex}
A continuously differentiable function  $F :\mathbb{R}^d  \xrightarrow{} \mathbb{R} $ is said to be \emph{$\tau$-star-convex} for $\tau > 0$ if there exists $x_\star \in \arg \min F $ such that for all $x \in \mathbb{R}^d$
\begin{equation}
\label{admissible}
\tau \langle \nabla F(x),x-x_\star \rangle ~\geq~ F(x) - F(x_\star).
\end{equation}
\end{definition}
Note that, all convex functions are $1$-star-convex \cite{joulani17}. The condition becomes weaker as $\tau$ grows; conversely, if allowed, $0$-star convex function would be identically constant. In our unconstrained optimization setting, we observe that $x_\star$ must be a global minimizer (as the gradient vanishes at global minima).

For first-order methods, it is often assumed that the gradient is bounded, Lipschitz or $\alpha$-H\"older continuous \cite{NY83, Diakonikolas2024}, implying that the objective function satisfies certain smoothness conditions. In the following, we introduce the definition of (weak) smoothness with respect to a norm $\|\cdot\|$ using the notion of Bregman divergence, defined as follows.

\begin{definition}[Bregman divergence]\label{def:bregman}
For a continuously differentiable function $F :\mathbb{R}^d  \xrightarrow{} \mathbb{R} $, the \emph{Bregman divergence} from $x \in \mathbb{R}^d$ to $y \in \mathbb{R}^d$ is defined by
\[ D_F(x,y) := F(x) - F(y) - \langle \nabla F(y),x-y \rangle. \]
\end{definition}
Note that we do not assume convexity of $F$, and consequently we may have $D_F(x,y) < 0$.
\begin{definition}[Weak smoothness]
\label{def:weak smooth}
A continuously differentiable function  $F :\mathbb{R}^d  \xrightarrow{} \mathbb{R} $ is said to be \emph{$(L,\kappa)$-weakly smooth} with respect to norm $\|\cdot \|$ for $1< \kappa \leq 2 $ if for all $x,y \in \mathbb{R}^d$
\[ | D_F(x,y)| \leq \frac{L}{\kappa} \|x-y\|^{\kappa}.\]
The special case $\kappa=2$ is simply referred to as \emph{smooth}.
\end{definition}
 To develop our mirror descent algorithm, we use a so-called distance-generating function (see Assumption~\ref{ass}), which fulfils a strong or uniform convexity assumption with respect to an arbitrary norm, defined as follows.
\begin{definition}[Uniform convexity]
\label{def-uni-convex}
A continuously differentiable function  $F :\mathbb{R}^d  \xrightarrow{} \mathbb{R} $ is said to be \emph{$(\mu,q)$-uniformly convex} with respect to norm $\|\cdot \|$ for $q\geq 2$ if for all $x,y \in \mathbb{R}^d $
\[ D_F(x,y) \geq \frac{\mu}{q} \|x-y\|^{q}.\]
The special case $q=2$ is called \emph{strongly convex}.
\end{definition}
The classic acceleration setting considers $q=\kappa=2$. In general, we have $q > 2 > \kappa$.

\section{Main results}\label{sec3}
Our accelerated optimization algorithm for weakly smooth, star-convex functions incorporates both a mirror descent step and a proximal step, similar to~\cite{AGJ18}. However, as we will see in the proof sketch, our analysis exploits the specific structure of unconstrained optimization.  We start by summarizing our assumptions. 
\begin{assumption}\label{ass}
We have a differentiable function $F : \mathbb R^d \to \mathbb R$ to be optimized, a norm $\norm{\cdot}$ defined on $\mathbb R^d$, a function $\psi : \mathbb R^d \to \mathbb R$ customarily called the distance-generating function, an order $q \ge 2$ and constants $\tau, L, \kappa, \mu$ such that
\begin{itemize}
\item $F$ is $\tau$-star-convex (Definition~\ref{def star-convex}).
\item $F$ is $(L, \kappa)$-weakly smooth w.r.t.\ $\norm{\cdot}$ (Definition~\ref{def:weak smooth}).
\item The
  distance-generating
  function $\psi$ is $(\mu,q)$-uniformly convex w.r.t.\ $\norm{\cdot}$ (Definition~\ref{def-uni-convex}).
\end{itemize}
\end{assumption}


For an arbitrary norm, constructing a distance-generating function that satisfies the uniform convexity assumption of Definition~\ref{def-uni-convex} requires careful consideration. Below we present examples of distance-generating functions for $p$-norms and composite norms, each meeting the required conditions of uniform convexity.

\begin{example}[$p$-norms]\label{p-norm example}
 To construct the corresponding distance-generating function, we define $\psi(x) = \frac{1}{q} \| x \|_p^q$, where the appropriate order $q$ depends on $p$. Specifically, for $p \in (1,2]$, we set $q=2$, ensuring that $\psi$ is $(p-1)$-strongly convex with respect to $\norm{\cdot}_p$ \cite{uniformconvex1,uniformconvex2}. For $p > 2$, we choose $q=p$, ensuring that $\psi$ is $(2^{-\frac{p(p-2)}{p-1}},p)$-uniformly convex with respect to \ $\|\cdot\|_p$ \cite[see][]{Zalinescu:1983,AGL24}.
 \\
 \\
For both cases, we will show sharp convergence guarantees of our algorithm, as summarized in Table~\ref{our:results}.
\end{example}

\begin{example}[Composite norms and distance generating function]
For composite norms, the distance-generating function is constructed by combining the distance-generating functions associated with the individual norms, ensuring the desired uniform convexity property. Consider, for instance, the following composite norm
\begin{equation} \label{ex-composite}
  \|x\|_{2 \circ 1.5}
  ~:=~
  \sqrt{\frac{1}{2} \|x_{1:d/2}\|^2_2+\frac{1}{2}\|x_{(d/2+1):d}\|^2_{1.5}}
  .
\end{equation}
By Example~\ref{p-norm example},
$\psi_1(x_{1:d/2}) = \frac{1}{2} \sum_{i=1}^{d/2} x_i^2$ is $1$-strongly convex with respect to $\|\cdot\|_2$ and 
$\psi_2(x_{d/2+1,d}) = \frac{1}{1.5} (\sum_{i=d/2 + 1}^{d} |x_i|^{1.5})^{2/1.5}$ is $\frac{1}{2}$-strongly convex with respect to $\|\cdot\|_{1.5}$. Hence, we define the composite distance-generating function as $\psi(x) :=  \frac{1}{2} \psi_1(x_{1:d/2}) + \frac{1}{2} \psi_2(x_{d/2+1,d}).$
For any $ x,y \in \mathbb{R}^d$, we denote the Bregman divergence by $D_{\psi}(x,y)$ and we note that
\begin{align*}
 D_{\psi}(x,y)  & = \frac{1}{2} \Big( \psi_1(x_{1:d/2}) - \psi_1(y_{1:d/2}) - \langle \nabla \psi_1(y_{1:d/2}),x_{1:d/2} - y_{1:d/2} \rangle \Big) \\
& \quad + \frac{1}{2} \Big( \psi_2(x_{(d/2+1):d}) - \psi_2(y_{(d/2+1):d}) - \langle \nabla \psi_2(y_{(d/2+1):d}),x_{(d/2+1):d} - y_{(d/2+1):d} \rangle \Big) \\
 & \geq \frac{1}{4} \|x_{1:d/2} - y_{1:d/2} \|_2^2 + \frac{1}{8} \| x_{(d/2+1):d} - y_{(d/2+1):d} \|_{1.5}^2  ~\geq~ \frac{\| x-y \|_{2 \circ 1.5}^2}{4}.
\end{align*}
We use the strong convexity of $\psi_1$ and $\psi_2$ in the last line. Thus, $\psi$ is $\frac{1}{2}$-strongly convex with respect to the composite norm $\| \|_{2 \circ 1.5}$. In general, the strong convexity coefficient  of the composite norm is determined by the weaker of the two individual norms.

\end{example}
Our proposed method, outlined in Algorithm~\ref{algo-acc}, extends the accelerated method of~\cite{pmlr-v125-hinder20a}, which addresses the case of smoothness ($\kappa=2$) for the Euclidean norm $\norm{\cdot}_2$. This norm is known to be strongly convex with respect to itself, with $q=2$ and $\psi(\cdot) = \frac{1}{2}\norm{\cdot}_2^2$.
Similar to their method, ours maintains an auxiliary sequence $(x_t)_{t \ge 1}$ of iterates, from which it derives a sequence $(x_t^{md})_{t \ge 1}$ of gradient query points as well as the output sequence $(x_t^{ag})_{t \ge 1}$ of approximate minimizers. In each iteration, the gradient point $x_t^{md}$ is found by binary search, which we extend from~\cite{pmlr-v125-hinder20a} to handle weak smoothness in Section~\ref{sec4}. While this search is not necessary for convex $F$, it appears to be a reasonable trade-off for the relaxation to star-convexity, see also the discussion in Section~\ref{sec5}. \cite{pmlr-v125-hinder20a} update the iterates $x_t$ and approximate minimizers $x_t^{ag}$ each using gradient descent. However, to address a non-Euclidean norm, we instead employ a mirror descent step with respect to the Bregman divergence $D_\psi$ to update $x_t$, and a proximal step with respect to $\frac{1}{q} \norm{\cdot}^q$ to update $x_t^{ag}$. The accelerated order $q \ge 2$ depends on the norm $\| \cdot \|$ and will asymptotically determine the convergence speed of our algorithm, with a smaller $q$ leading to faster rates.

Our algorithm takes as input the parameter schedules $(C_t, \epsilon_t,\alpha_t,\eta_t)$, with $(\alpha_t,\eta_t)$ being the step sizes and $(C_t, \epsilon_t)$ being the stopping conditions  for the binary search. All tuned parameters are polynomial functions of the horizon $T$, and are given in Theorem \ref{thm-p>2} and Appendix~\ref{convergence proof}. 
\begin{algorithm}
\caption{Accelerated Mirror Descent with Binary Search (for the setting of  Assumption~\ref{ass})}
\label{algo-acc}
\begin{algorithmic}[1]
  \Require
  Starting point $x_1 = x_1^{ag} \in \mathbb{R}^d$, iterations $T \geq 0$, parameters $(C_t, \epsilon_t,\alpha_t,\eta_t)_{t \ge 1}$ \\
\textbf{For} {$ 1 \leq t \leq T$} \textbf{do} \\
\label{stepbinary}
\quad \quad \textit{Binary search} (Alg.~\ref{algo-binary-search}) with endpoints $(x_t,x_t^{ag})$ and parameters $(C_t,\epsilon_t)$ to find $\lambda_t \in [0,1]$, $x_{t}^{md} = \lambda_t x_{t}^{ag} + (1-\lambda_t) x_{t}$ such that $ \langle \nabla F(x_t^{md}),x_t^{md} - x_{t}  \rangle
+  C_t  \Big( F(x_t^{md}) - F(x_t^{ag})  \Big) \le   \epsilon_t $ \\
\quad \quad Compute $ x_{t+1}  = \arg \min_{x \in \mathbb{R}^d} \{ \eta_t \langle \nabla F(x_{t}^{md}), x  \rangle +  D_\psi(x,x_t) \} $ 
\label{stepmid}
\\
\label{stepagg}
\quad \quad Compute $ x_{t+1}^{ag}  = \arg \min_{x \in \mathbb{R}^d}  \{ \alpha_t \langle \nabla F(x_{t}^{md}), x  \rangle +  \frac{\mu}{q} \| x - x_t^{md} \|^q \}   $ \\
\textbf{EndFor} \\
\textbf{Return} $x_{T+1}^{ag}$
\end{algorithmic}
\end{algorithm}

As we will see in Theorem~\ref{thm-binary-search}, the total number of values and gradients queried in each binary search step is $\mathcal{O}(\log T)$. In turn, the total number of gradients queried in $T$ rounds is $\mathcal{O}(T \log T)$. We note that the optimal tuning of step sizes will depend on the number of iterations $T$.


\subsection{Convergence result}
In this section, we present the convergence result of our algorithm. The full proof is provided in Appendix~\ref{convergence proof}, with a proof sketch given in Section~\ref{sketch}.
\begin{theorem}
\label{thm-p>2}
In the setting of Assumption~\ref{ass} with $\kappa < q$, Algorithm~\ref{algo-acc} with the tuning below returns $x_{T+1}^{ag}$ after $T$ iterations with precision
\[ F(x_{T+1}^{ag}) - F_\star
  ~=~
  \mathcal{O}_{q,\kappa} \left(
    \frac{L  \tau^q}{T^{\frac{\kappa q + \kappa - q}{q}  }}
    \left(
    \frac{  D_{\psi}( x_\star ,x_1) }{\alpha L }
    +
    \Big( \frac{\alpha L}{\mu}\Big)^{\frac{\kappa}{q-\kappa}} \log (T)
  \right)
  \right)
\]
where $\mathcal{O}_{q, \kappa}$ omits constants depending only on $(q,\kappa)$. This is achieved for any $\alpha > 0$ by
\[
 \begin{aligned}
\alpha_t & :=    \Big(\tau (q-\beta) \Big)^{q-\kappa} \frac{\alpha}{t^\beta}, \quad \quad   \eta_t := \alpha_t \Big(\frac{t}{\tau(q-\beta)} \Big)^{q-1}.  \\
C_t & :=   \Big( \frac{\eta_t}{\alpha_t}\Big)^{q_{\ast}-1}  - \frac{1}{\tau} = \frac{1}{\tau} \Big( (q-\beta) t - 1 \Big)\\
\epsilon_t & :=  \frac{\alpha^{\frac{q}{q-\kappa}}  }{ t \eta_t }  \frac{M^{1/r} }{ \mu^{1/r}} L ^{1+1/r}
.
\end{aligned}
\qquad
\text{where}
\qquad
\begin{aligned}
\begin{cases}
r & := \frac{q-\kappa}{\kappa} > 0 \\
M & :=   \big( \frac{r}{q} \big)^{r}  \geq 0 \\
\beta & =  q-\kappa + \frac{q-\kappa}{q}
\end{cases}
\end{aligned}
\]
\end{theorem}

If $q$ and $\kappa$ are close to $2$, our rate is close to the classic $T^{-2}$ accelerated rate by~\cite{NEMIROVSKII198521}. The case $q=\kappa=2$ requires special handling, as the above rate would yield an infinite exponent in the last term. We introduce a refined tuning for this case along with a tighter analysis in Appendix~\ref{app:qk2}. 

\textit{Tuning the parameters.} The general step sizes $\alpha_t$ and $\eta_t$ are polynomial in the iteration number $t$ with the main constraint being that their exponent must remain bounded. Specifically, we set $\alpha_t \sim \alpha t^{-\beta}$, where $\alpha > 0$ is a coefficient determined based on the Bregman divergence $D_{\psi}(x_\star,x_1)$. Additionally, $\eta_t$ is chosen as $\eta_t \sim t^{1 - \frac{3 q - \kappa (1+q)}{q}}$, which recovers the classic Lipschitz convex setting $\eta_t \sim 1/\sqrt{t}$ for $q=2,\kappa=1$ and the Nesterov step size of $\eta_t \sim t$ for $q=\kappa=2$. The choice of $C_t \sim t$ follows from the mirror descent analysis, while $\epsilon_t$ is chosen to decrease sufficiently fast such that $\eta_t \epsilon_t \lesssim 1/t$ sums to $\mathcal O(\log T)$. Full details are provided in Appendix~\ref{convergence proof}. 

\textit{Tuning the parameter $\alpha$.} The optimal tuning for $\alpha$ in the preceding result depends on the radius $D_\psi(x_\star, x_1)$, which is unknown a priori. It also depends on $\log T$, while we prefer tuning independent of the time horizon $T$. To avoid both issues, we suggest the following practical tuning.

\begin{corollary}
  \label{corr-p>2}
  If a bound $\frac{1}{\mu} D_{\psi}( x_\star ,x_1) \le  B$ is available, then the tuning of Theorem~\ref{thm-p>2} with
  $\alpha
  ~=~
  \frac{\mu}{L}
  \left(
  \frac{
    (q-\kappa) B
  }{
    \kappa 
  }
\right)^{\frac{q-\kappa}{q}}$
guarantees
    $F(x_{T+1}^{ag}) - F_\star
    ~=~
    \mathcal{O}_{q,\kappa}  \left(
      L
      \tau^q
      B^{\frac{\kappa}{q}}
      \frac{
        \log (T)}{T^{\frac{\kappa q + \kappa - q}{q}  }}
    \right)
    $.
\end{corollary}

In the next sections, we will sketch the proof of Theorem~\ref{thm-p>2} and examine the complexity of binary search in Section \ref{sec4}. After that, we will finish the complexity analysis of our algorithm in Section~\ref{sec-application}.

\subsection{Proof Sketch of Theorem~\ref{thm-p>2} (Convergence rate of Algorithm~\ref{algo-acc})}\label{sketch}
The proof can be decomposed into three steps. As the first step of the proof, we use star-convexity of the objective $F$ to linearize it around the gradient point $x_t^{md}$. Together with the mirror descent iterates $x_{t+1}$ from line~\ref{stepmid} of Algorithm~\ref{algo-acc}, we obtain the upper bound
\begin{equation}
\label{eq-acc-potential}
\begin{aligned}
 \frac{\eta_t}{\tau} \Big( F(x_t^{md}) - F_\star \Big) 
{}\leq{} &  D_{\psi}(x_\star ,x_t) -  D_{\psi}(x_\star,x_{t+1}) +  \eta_t \langle \nabla F(x_t^{md}),x_t^{md} - x_t  \rangle \\
{}+{} & \eta_t \langle \nabla F(x_t^{md}),x_{t} - x_{t+1}  \rangle  - \frac{\mu}{q} \| x_{t+1} - x_{t} \|^q.
\end{aligned}
\end{equation}
As the next step, we consider the unconstrained proximal step in line~\ref{stepagg} of Algorithm~\ref{algo-acc}, i.e.,
\[ x_{t+1}^{ag}  = \arg \min_{x \in \mathbb{R}^d}~  \Big\{ \alpha_t \langle \nabla F(x_{t}^{md}), x  \rangle +  \frac{\mu}{q} \| x - x_t^{md} \|^q \Big\}. \]
This proximal step leads to the following gap between two neighbouring accelerated potentials $\left( F(x_{t}^{md}),F(x_{t+1}^{ag}) \right)$ 
\begin{align*}
 \eta_t \langle  \nabla F(x_t^{md}) , x_{t} - x_{t+1} \rangle  - \frac{\mu}{q} \| x_{t+1} - x_{t} \|^q  \leq   A_t \Big( F(x_t^{md}) - F(x_{t+1}^{ag}) \Big)   +    B_t 
\end{align*}
with parameters $A_t,B_t$ specified in \eqref{notation-proof-p>2}. The term $B_t$ is the classical residual after combining the weak smoothness assumption of $F$ with the inexact gradient-trick \cite{CT93, AGJ18, AGL24}. This step was previously known to hold in the convex case \cite{AGJ18} or in the unconstrained Euclidean case \cite{pmlr-v125-hinder20a} for different reasons. For the non-convex non-Euclidean unconstrained case, one needs to examine carefully the regularity of the norm square (see Appendix~\ref{app:prelim.proofs}).

After computing the new accelerated point $x_{t+1}^{ag}$ from the middle point $x_t$, the previous objective bound \eqref{eq-acc-potential} can be manipulated into an upper bound on the potential gap between two iterations $\left( F(x_{t}^{ag}),F(x_{t+1}^{ag}) \right)$. 
\begin{equation}
\begin{aligned}
 A_t   \Big( F(x_{t+1}^{ag}) - F_\star \Big) & \leq   D_{\psi}(x_\star ,x_t) -  D_{\psi}(x_\star,x_{t+1}) + \eta_t \langle \nabla F(x_t^{md}),x_t^{md} - x_{t}  \rangle \\
& +  \Big(   A_t  -  \frac{\eta_t}{\tau} \Big) \Big( F(x_t^{md}) - F(x_t^{ag})  \Big)   +  \Big(   A_t  -  \frac{\eta_t}{\tau} \Big) \Big( F(x_t^{ag}) - F_\star  \Big)  +   B_t.
\end{aligned}
\end{equation}
As the last step, the choice of $\lambda_t$ in line~\ref{stepbinary} of Algorithm~\ref{algo-acc} ensures that
\begin{align*}
\eta_t \langle \nabla F(x_t^{md}),x_t^{md} - x_{t}  \rangle  +   \Big(   A_t  -  \frac{\eta_t}{\tau} \Big) \Big( F(x_t^{md}) - F(x_t^{ag})  \Big) \leq \eta_t \epsilon_t. 
\end{align*}
By showing that our tuning satisfies $A_{t} -A_{t-1} \leq \frac{\eta_t}{\tau}$, $ \eta_t \epsilon_t = \frac{\mu}{t}$ and $B_t \approx \frac{1}{t}$, we get 
\begin{equation}
\begin{aligned}
 A_t   \Big( F(x_{t+1}^{ag}) - F_\star \Big) \leq   D_{\psi}(x_\star ,x_t) -  D_{\psi}(x_\star,x_{t+1}) +  A_{t-1} \Big( F(x_t^{ag}) - F_\star  \Big)  + O(1/t),
\end{aligned}
\end{equation}
which telescopes to our result. To obtain the desired accelerated convergence rate in Table~\ref{our:results}, we still need to account for the complexity of finding $\lambda_t$. In Section~\ref{sec4}, we show that this additional cost remains only a logarithmic factor, even in the non-Euclidean case. 

\section{Binary search for H\"older smooth functions}
\label{sec4}
 In each round $t$, Algorithm~\ref{algo-acc} performs binary search to find the momentum parameter $\lambda_t$ that satisfies
\begin{equation}
\label{bin_cond}
\lambda_t g'(\lambda_t)+ C_t g(\lambda_t) \leq \epsilon_t, 
\end{equation}
where we use the abbreviation $g(\lambda) \triangleq F(\lambda x_t^{ag}+(1-\lambda)x_t) - F(x_t^{ag}) $ so that in particular $g(1) = 0$. The condition~\eqref{bin_cond} is crucial for the convergence proof of Theorem~\ref{thm-p>2}. The core idea behind our Algorithm~\ref{algo-binary-search} is inspired by Algorithm 2 of~\cite{pmlr-v125-hinder20a}. We aim to efficiently find a $\lambda_t$ that satisfies \eqref{bin_cond} in the more general setting where $F$ is weakly smooth with respect to a possibly non-Euclidean norm. 

Note that if $g(0) \leq 0$ or $g$ is decreasing at $\lambda=1$, then \eqref{bin_cond} is immediately satisfied. If neither of these conditions hold, then $g(0)>g(1)$ and $g'(1)>0$, indicating that $g$ must switch at least once from decreasing to increasing. For the largest $\lambda_*$ fulfilling $g'(\lambda_*)=0$ it holds that $g(\lambda_*) < 0$. We further show in the proof of Theorem~\ref{thm-binary-search} that all $\lambda \in [\lambda_*-\delta, \lambda_*]$, with a properly chosen $\delta$, must satisfy condition \eqref{bin_cond}. The efficiency of our algorithm in finding any $\lambda$ in that interval is stated in Theorem~\ref{thm-binary-search} below.

\begin{algorithm}[h!]
  \caption{Generalized binary search}
  \label{algo-binary-search}
\begin{algorithmic}
\Require Objective function $F$, endpoints $(x_t, x_t^{ag})$, parameters $(C_t, \epsilon_t) \geq 0 $\\
Define $g(\lambda) \triangleq F(\lambda x_t^{ag}+(1-\lambda)x_t) - F(x_t^{ag})$ \\
\textbf{if} $ g'(1) \leq \epsilon_t$ \\
\quad \quad \textbf{Return} $(\lambda_t =1, x_{t}^{md} = x_t^{ag})$ \\
\textbf{else if} $C_t g(0) \leq \epsilon_t $ \\
\quad \quad \textbf{Return}  $(\lambda_t =0, x_{t}^{md} = x_t)$\\
Initialize $ [a,b] = [0,1]$ \\
\textbf{loop} \\
\quad \quad Set $\lambda_t = \frac{a+b}{2}$ \\
\quad \quad \textbf{if} $ g'(\lambda_t)+  C_t  g(\lambda_t)  \leq   \epsilon_t $ \\
\quad \quad \quad \quad \textbf{Return}$(\lambda_t, x_{t}^{md} = \lambda_t x^{ag}_{t} + (1-\lambda_t) x_{t})$ \\
\quad \quad \textbf{else if} $  g(\lambda_t) > 0 $ \\
\quad \quad \quad \quad $a = \lambda_t$ \Comment{iterate with $[\lambda_t,b]$} \\
\quad \quad \textbf{else} \\
\quad \quad \quad \quad $b = \lambda_t$ \Comment{iterate with $[a,\lambda_t]$} \\
\textbf{end loop}
\end{algorithmic}
\end{algorithm}

\begin{theorem}
\label{thm-binary-search}
Under Assumption~\ref{ass} (in particular $F$ is $(L,\kappa)$-weakly smooth), and for fixed stopping parameters $(C_t,\epsilon_t)$ and endpoints $(x_t,x_t^{ag})$,  Algorithm~\ref{algo-binary-search} finds $\lambda$, satisfying
\begin{align*}
x_{t}^{md} & = \lambda x_{t}^{ag} + (1-\lambda) x_{t}  \\
 \langle \nabla F(x_t^{md}),x_t^{md} - x_{t}  \rangle 
& +  C_t  \Big( F(x_t^{md}) - F(x_t^{ag})  \Big) \leq \epsilon_t  
\end{align*}
in $\mathcal{O}_{\kappa}\left(\max\{ \log(C_t),  \log \Big( \frac{ L\|x_t - x_t^ {ag} \|^\kappa }{ \epsilon_t }\Big) \} + 1 \right) $ iterations.
\end{theorem}

\cite{pmlr-v125-hinder20a} analyze their binary search algorithm for smooth functions, i.e., specifically with $\kappa=2$ in Definition~\ref{def:weak smooth}. Our algorithm accommodates a general $\kappa$, addressing the broader class of weakly smooth functions. While the running time depends on $\kappa$, it appears only as an exponent in the logarithm. Consequently, this introduces only a constant multiplicative factor, thereby generalizing the algorithm of~\cite{pmlr-v125-hinder20a}.
\\
\\
{\bf Proof Sketch.}
The proof can be roughly divided into \textit{three} parts. We first show that there exists a $\lambda_*$ satisfying the strengthening of condition \eqref{bin_cond} with $\epsilon_t=0$. Next, we show that  this implies that condition \eqref{bin_cond} holds for all $\lambda \in [\lambda_*-\delta, \lambda_*]$, where $\delta$ is properly chosen depending on the weak smoothness parameters $(L,\kappa)$ of $F$ and the parameters $(C_t,\epsilon_t,x_t,x_t^{ag})$ of the algorithm. Finally we prove that our algorithm finds a $\lambda$ within that interval after at most $\mathcal{O}\Big(\log\Big( \frac{L C_t \| x_t - x_t^{ag} \|^\kappa }{\epsilon_t }\Big)\Big)$ iterations (details in Appendix~\ref{appendix-binary-search-proof}).

\subsection{Polynomial growth of our iterates}

To complete our analysis, we present upper bounds for the iterates $(x_t,x_t^{ag})$ generated by Algorithm~\ref{algo-acc}. In fact, our Theorem~\ref{thm-binary-search} shows that the number of iterations is upper bounded by a term logarithmic in $t$ provided the distance $\| x_t - x_{t}^{ag}\|$ grows at most polynomially in $t$. This result is formalized in the following theorem.

\begin{theorem}
\label{bounded-iterate}
Let $D_{\psi}(x_\star,x_1) \leq B $. Under Assumption~\ref{ass}, suppose we are running Algorithm~\ref{algo-acc}
such that for all $t \geq 1$
\begin{align*}
\max \left(  \Big(\frac{qL\eta_t}{\kappa} \Big)^{\frac{1}{q-1}} , \Big( \frac{L\alpha_t}{\kappa} \Big)^{\frac{1}{q-1}}   \right)\leq  K^{n_1}  t^{n_2},
\end{align*}
where $n_1,n_2 \geq 0$, and $K^{n_1} > B$ with $K^{n_1}$ potentially depending on $(L,\tau,\mu)$. Define the iterate radius by $R_t = \max\{ \| x_t - x_\star\|, \| x_t^{ag} - x_\star \|, \| x_t^{md} - x_\star \|\}$. Then, for all $t \geq 1$, we have
\[ R_t = \mathcal{O} \Big( K^{\frac{n_1(q-1)}{q-\kappa}} t^{\frac{(q-1)(n_2 + 1)}{q-\kappa}}\Big).\]
\end{theorem}

\subsection{Lower bound for Binary search}
Although the following theorem does not contribute to our upper bounds, we analyze the binary search algorithm without assuming star-convexity to highlight the efficiency of Algorithm~\ref{algo-binary-search} under weak smoothness. This section shows that the number of iterations needed in Theorem~\ref{thm-binary-search} to find a suitable $\lambda$ satisfying condition \eqref{bin_cond}  is both sufficient and necessary - fewer itations could yield in a $\lambda$ that does not satisfy condition \eqref{bin_cond}. To show this lower bound, we determine the minimum number of iteration steps required to find a suitable $\lambda$ for a univariate smooth target function.


\begin{theorem} 
Let $C,\epsilon, L_\star > 0$. Consider a sequence of $N$ points $(\lambda_1,\ldots,\lambda_N)$ evaluated at each iteration $(g(\lambda_i),g'(\lambda_i))$. If the number of iterations is insufficient meaning 
\[ N < \log(5) \Big( \log \frac{L_\star}{\epsilon} + \log \frac{C}{(C+1)^2} - \log(88) \Big),  \]
no algorithm can find $\lambda$ fulfilling condition \eqref{bin_cond}.
\end{theorem}
The theorem shows that if $N \lesssim  \log(L_\star/\epsilon) $,  it becomes impossible to determine $\lambda$ with precision $\epsilon$. The full version of this lower bound is given in Theorem \ref{induced-lower-bound-binary}.

\section{Summary}
\label{sec-application}
By combining the convergence results from Corollary~\ref{corr-p>2} with the complexity analysis presented in Theorem~\ref{thm-binary-search}, we establish a bound on the precision error of Algorithm~\ref{algo-acc} in terms of the number of first-order oracle calls.
\begin{corollary}
  \label{corr-application} Let $T > 0$ be the number of iterations and $\frac{1}{\mu} D_{\psi}( x_\star ,x_1) \le  B$. The tuning of Theorem~\ref{thm-p>2} with
  $\alpha
  ~=~
  \frac{\mu}{L}
  \left(
  \frac{
    (q-\kappa) B }{
    \kappa 
  }
\right)^{\frac{q-\kappa}{q}}$
guarantees
$F(x_{T+1}^{ag}) - F_\star
    ~=~
    \mathcal{O}_{q,\kappa}  \left(
      L
      \tau^q
      B^{\frac{\kappa}{q}}
      \frac{
        \log (T)}{T^{\frac{\kappa q + \kappa - q}{q}  }}
    \right).$
The algorithm's oracle usage over $T$ iterations is upper bounded by $\mathcal{O}(T \log(LB\tau T))$, where this bound represents the maximum number of times the oracle is called during the entire execution process.
\end{corollary}

To apply Corollary~\ref{corr-application} to the $p$-norms in Example~\ref{p-norm example}, we consider a problem where both the initial points $x_1$ and minimizer $x_\star$ lie within a $p$-norm ball of radius $R$, i.e., \ $\norm{x_1}_p, \norm{x_\star}_p \le R$, which implies that $\norm{x_1 - x_\star}_p \le 2R$ by the triangle inequality. For $p$-norms, recall that we use the distance-generating function $\psi(\cdot)=\frac{1}{q} \norm{\cdot}_p^q$ with $q = \max\{2,p\}$. In Appendix~\ref{bregman vs radius}, we show that the Bregman divergence also satisfies $D_\psi(x_\star, x_1) \le 2 R^q$. Combining Theorems~\ref{thm-p>2} and~\ref{bounded-iterate}, we obtain the following convergence guarantees.
\begin{corollary}
  \label{cor-p<2}
  Consider the setting of Assumption~\ref{ass} and let $\norm{\cdot}_p$ denote the $p$-norm.
\begin{itemize}
    \item[(i)] 
    For $1 < p \leq 2$ and $\kappa <2$, Algorithm~\ref{algo-acc} with $q=2$ returns an $\epsilon$-optimal solution within
\[  \mathcal{O}_{\kappa} \left(\left(\frac{L \tau^{2} R^\kappa}{\epsilon}\right)^{\frac{2}{3\kappa - 2 }} \log^2 \Big( \frac{L \tau R}{\epsilon} \Big) \right)  \]
evaluations of the function value and gradient.
\item[(ii)] \label{cor-p}
For $ p > 2$, Algorithm~\ref{algo-acc} with $q=p$ returns an $\epsilon$-optimal solution within
\[ \mathcal{O}_{p,\kappa} \left(\Big(\frac{L \tau^{p} R^\kappa}{\epsilon}\Big)^{\frac{p}{\kappa p+\kappa-p}}  \log^2 \Big(\frac{L \tau R}{\epsilon} \Big) \right)   \]
evaluations of the function values and the gradients.
\end{itemize}
\end{corollary}

Note that the convergence guarantees of Corollary~\ref{cor-p<2} for star-convex functions match the lower bounds proved for convex functions, completing Table~\ref{our:results}. For $p=1$, our analysis applies with a dimension factor, detailed in Appendix~\ref{appendix-p=1}. For the smooth case $\kappa=2$ with $1<p<2$, where iterates may grow exponentially, we extend our algorithm while maintaining similar complexity; further details are provided in Appendix~\ref{app:qk2}. 

\section{Conclusion and future research}
\label{sec5}
In this work, we introduce a novel class of structured non-convex functions, namely $\tau$-star-convex functions that are weakly smooth  with respect to an arbitrarily norm. Our framework applies to a broad subclass of star-convex functions and generalizes the setting analyzed in
\cite{pmlr-v125-hinder20a} which focuses on smooth star-convex functions with respect to the Euclidean norm.

We develop the accelerated Algorithm~\ref{algo-acc} to efficiently minimize any weakly smooth function within this broad class, building on and extending the binary search technique from~\cite{pmlr-v125-hinder20a}. Our analysis shows that the algorithm achieves a near-optimal accelerated convergence rate for all $p$-norms with $p>1$. Additionally, our algorithm has a runtime dependence on $\tau^{\max\{2,p\}}$ for $p\geq 1$, shown to be optimal in the Euclidean case when $p=2$, see \cite{pmlr-v125-hinder20a}.

While our work focuses on unconstrained minimization problems, future research could extend our approach to constrained settings. Currently, the proximal step for calculating  $x_{t+1}^{ag}$ in Algorithm~\ref{algo-acc} is shown to be effective only for unconstrained problems. For the constrained non-convex case, since evaluating the direction of the constrained proximal step becomes a challenge, it is unknown if accelerated convergence rates are possible. 


Finally, our method operates within a first-order access model, meaning that the algorithm has access to both the objective function value and gradient at the evaluation points. It remains an open problem whether this can be extended to ``pure'' first-order access models, that rely solely on gradient evaluations. While this extension is feasible in the convex case, it is unclear whether it can be achieved for star-convex functions. Another promising direction is to explore the applicability of our method in a stochastic model. Although stochastic accelerated mirror descent has been previously analyzed in the literature \cite{NIPS2009_ec5aa0b7}, the main challenge lies in adapting the binary search algorithm to the noisy setting.

\printbibliography

\newpage
\appendix

\section{A regularity property of norms}
\label{app:prelim.proofs}

In this section, we will discuss about the regularity property for norms which is core for our acceleration scheme. For a norm $\| \cdot \|$ with associated dual norm $\| \cdot \|_\ast $, we define the gradient (or any sub-gradient) of the squared norm $\|\cdot\|^2$ by
\[ \phi(x)
  ~:=~
  \nabla \left(  \frac{\|x\|^2}{2} \right)
  ~=~
  \operatorname*{argmax}_{y : \|y\|_\ast = \|x\|}~ y^\top x
  .
\]
Wherever the norm is not differentiable, we may pick $\phi(x)$ to be any sub-gradient or maximizer. To see the equality, recall that $\|x\| = \max_{z : \|z\|_\ast = 1} z^\top x$ and let $z$ attain that maximum. Then differentiation\footnote{Recall that a sub-gradient of a maximum $x \mapsto \max_y f(x,y)$ is the derivative of the objective $\nabla_x f(x,y_*)$ with a maximizer $y^*$ held fixed.} gives $\phi(x) = \|x\| z$. Reparametrising by $y = \|x\| z$ gives the right hand side.

\begin{fact} [Regularity of a norm]
\label{def-regular-norm}
For all $x \in \mathbb{R}^d $
\begin{align*}
\langle \phi(x), x\rangle & ~=~  \|x\|^2,  \\
\| \phi(x) \|_{\ast} & ~=~ \|x \|.
\end{align*}
\end{fact}
These two properties are crucial in our accelerated framework (Algorithm~\ref{algo-acc}), especially for weakly smooth objective functions. In particular, they allow for the choice of aggressive step-sizes and the derivation of accelerated convergence rates.


\begin{lemma}
\label{proof-regular-lemma}
For $s \geq 1 $, we define $\phi_{s}(x) := \nabla \Big( \frac{\|x\|^s}{s} \Big). $ For all $\alpha \ge 0$
\begin{align*}
\begin{cases}
\langle \phi_{s}(x), x\rangle & =  \|x\|^s  \\
\| \phi_s (x) \|_{\ast}^{\alpha} & = \|x \|^{\alpha(s - 1)}.
\end{cases}
\end{align*}
\end{lemma}

\begin{proof}
Recall that $\phi(x)= \nabla (\frac{\|x \|^2}{2})$. Then 
\begin{align*}
\phi_{s}(x) & = \nabla \Big( \frac{(\|x\|^2)^{s/2}}{s} \Big) = \frac{1}{s} \cdot \frac{s}{2} (\|x\|^2)^{\frac{s}{2}-1} \cdot \nabla \Big( \|x \|^2 \Big)
& =  (\|x\|^2)^{\frac{s}{2}-1} \cdot \frac{1}{2} \nabla \Big( \|x \|^2 \Big)  = \frac{ \phi(x)}{ \|x\|^{2-s}}.
\end{align*}
As 
\begin{align*}
  \langle \phi_{s}(x), x\rangle & = \frac{\langle \phi(x), x\rangle }{\|x\|^{2-s}} = \frac{\|x\|^2}{\|x\|^{2-s}} = \|x\|^s
  ,
  \\
  \| \phi_s(x) \|_\ast ^{\alpha} & = \frac{\| \phi(x) \|_\ast^{\alpha}}{\|x\|^{\alpha(2-s)}} = \frac{\|x \|^{\alpha}}{\|x\|^{\alpha(2-s)}}  = \|x \|^{\alpha(s - 1)}
  ,
\end{align*}
the assertion follows.
\end{proof}

\section{Useful lemma for convergence proof}

\begin{lemma}[Young's inequality, \cite{young}]
  \label{young-ineq}
Let $a,b \geq 0$ and $p,q \geq 1$ such that $\frac{1}{p} + \frac{1}{q} = 1$. Then it holds, that  
\[ ab \leq \frac{a^p}{p} + \frac{b^q}{q}.\]
\end{lemma}

\begin{lemma}  \label{lem:prox}
Let $f$, $\nu$ be convex functions and let $\nu$ be continuously differentiable. For \[ u^{\star} = \arg\min_{u\in {\cal X}} \{f(u) + D^{\nu}(u,y)\}, \]
and all $u$, it holds
\[f(u^{\star}) + D^{\nu}(u^\star,y) + D^{f}(u,u^\star) \leq f(u) +  D^{\nu}(u,y) - D^{\nu}(u,u^\star). \]
\end{lemma}

\begin{proof}
As $u^\star$ is a minimizer, it fulfills for all  $u\in{\cal X}$ the first-order optimality condition
\begin{align}
\label{foc}
\langle \nabla f(u^\star) + \nabla D^{\nu}(u^{\star},y) , u - u^{\star}\rangle \geq 0,
\end{align}
 with the gradient of the divergence $D^\nu$ taken with respect to the first argument. Applying the three-points identity \cite{CT93}, yields
\[\langle \nabla D^{\nu}(u^{\star},y) , u - u^{\star}\rangle = D^{\nu}(u,y)- D^{\nu}(u,u^{\star})- D^{\nu}(u^{\star},y),\]
leading to
\begin{align*}
  f(u) - f(u^{\star}) - D^{f}(u,u^{\star}) & = \langle \nabla f(u^\star) , u - u^{\star}\rangle \geq D^{\nu}(u,u^{\star})- D^{\nu}(u,y)  +D^{\nu}(u^{\star},y).
\end{align*}

\end{proof}

\begin{lemma}
\label{Lemma inexact gradient}
Set
\begin{equation}
    \begin{aligned}
    \begin{cases}
    r & := \frac{q-\kappa}{\kappa} > 0 \\
    M & :=   \big( \frac{r}{q} \big)^{r}  \geq 0.
    \end{cases}
    \end{aligned}
\end{equation}
Then for all $\delta>0,x,y \in \mathbb{R}^d $
\begin{equation}
\begin{aligned}
\frac{\| x-y \|^\kappa}{\kappa} \leq  \frac{M}{q \delta^r} \|x-y\|^q +  \delta .
\end{aligned}
\end{equation}
\end{lemma}

\begin{proof}
  Using Young's inequality~\ref{young-ineq} with $\frac{1}{a} + \frac{1}{b} = 1 $, leads to $t \leq \frac{1}{az}t^a + \frac{1}{b}z^{b-1}.$  By choosing $t:= \frac{\|x-y\|^\kappa}{\kappa}$, $a=\frac{q}{\kappa}$, $b=\frac{q}{q-\kappa}$, $z= (b\delta)^{\frac{1}{b-1}}= (b\delta)^{r}$, it follows that
\[ \frac{\|x-y\|^{\kappa} }{\kappa}\leq  \frac{M}{q \delta^r} \|x-y\|^q +  \delta, \]
where we use that
\[\frac{\kappa }{q} \left( \frac{q-\kappa}{q \delta}\right)^{r} \left(\frac{1}{\kappa}\right)^{q/\kappa} = \frac{1}{q \delta^r}\left(\frac{q-\kappa}{q \kappa}\right)^{r} = \frac{M}{q \delta^r}. \]
\end{proof}

\section{Convergence proof of Theorem~\ref{thm-p>2} }
\label{convergence proof}
Our proof can be decomposed in \textit{three} major steps. We first derive a bound for $F(x_t^{md})-F_\star$ by analyzing our mirror descent step, then we proceed with an acceleration step comparing $F(x_{t+1}^{ag})$ and $F(x_t^{md})$ and finally we tune the parameters to satisfy the required conditions.

\subsection{Mirror descent analysis}

We apply Lemma~\ref{lem:prox} to the linear function $f(x) = \eta_t \langle \nabla F(x_t^{md}),x\rangle$ and set $ \nu = \psi$,  $u^\star = x_{t+1}$, $y=x_t$, $u = x$. For all $x \in \mathbb{R}^d$, this yields
\[ \eta_t \langle \nabla F(x_t^{md}),x_{t+1} \rangle + D_{\psi}(x_{t+1},x_t) \leq \eta_t \langle \nabla F(x_t^{md}),x \rangle +  D_{\psi}(x,x_t) -  D_{\psi}(x,x_{t+1}).
\]
By setting $x=x_\star$ and subtracting $\eta_t \langle \nabla F(x_t^{md}),x_\star \rangle $ on both sides, we obtain
\begin{align*}
  \eta_t \langle \nabla F(x_t^{md}),x_{t+1} - x_\star \rangle + D_{\psi}(x_{t+1},x_t) \leq  D_{\psi}(x_\star ,x_t) -  D_{\psi}(x_\star,x_{t+1})
  .
\end{align*}
We split
\[ \eta_t \langle \nabla F(x_t^{md}),x_{t+1} - x_\star \rangle = \eta_t \langle \nabla F(x_t^{md}),x_t^{md}  - x_\star \rangle + \eta_t \langle \nabla F(x_t^{md}),x_{t+1} - x_t^{md} \rangle, \]
and use star-convexity of $F$
\begin{align*}
\frac{\eta_t}{\tau} \Big( F(x_t^{md}) - F_\star \Big) \leq \eta_t \langle \nabla F(x_t^{md}),x_t^{md}  - x_\star \rangle
\end{align*}
to find
\begin{align*}
 \frac{\eta_t}{\tau} \Big( F(x_t^{md}) - F_\star \Big) + \eta_t \langle \nabla F(x_t^{md}),x_{t+1} - x_t^{md} \rangle + D_{\psi}(x_{t+1},x_t) \leq  D_{\psi}(x_\star ,x_t) -  D_{\psi}(x_\star,x_{t+1}).
\end{align*}
This leads to
\begin{equation}
\begin{aligned}
 & \frac{\eta_t}{\tau} \Big( F(x_t^{md}) - F_\star \Big) + D_{\psi}(x_{t+1},x_t) \\
 {}\leq{} &  D_{\psi}(x_\star ,x_t) -  D_{\psi}(x_\star,x_{t+1}) + \eta_t \langle \nabla F(x_t^{md}),x_t^{md} - x_{t+1}  \rangle  \\
{}={} &  D_{\psi}(x_\star ,x_t) -  D_{\psi}(x_\star,x_{t+1}) +  \eta_t \langle \nabla F(x_t^{md}),x_t^{md} - x_t  \rangle + \eta_t \langle \nabla F(x_t^{md}),x_{t} - x_{t+1}  \rangle
.
\end{aligned}
\end{equation}
By using that the distance-generating function $\psi$ is $(\mu,q)$-uniformly convex with respect to the norm $\|\cdot \| $, i.e., that it fulfills
\[D_{\psi}(x_{t+1},x_t) \geq \frac{\mu}{q} \| x_{t+1} - x_{t} \|^q,
\]
our equation becomes
\begin{equation}
\label{mirror_analysis}
\begin{aligned}
 & \frac{\eta_t}{\tau} \Big( F(x_t^{md}) - F_\star \Big)  \\
{}\leq{} &  D_{\psi}(x_\star ,x_t) -  D_{\psi}(x_\star,x_{t+1}) +  \eta_t \langle \nabla F(x_t^{md}),x_t^{md} - x_t  \rangle \\
 & {}+{}\eta_t \langle \nabla F(x_t^{md}),x_{t} - x_{t+1}  \rangle  - \frac{\mu}{q} \| x_{t+1} - x_{t} \|^q.
\end{aligned}
\end{equation}
The goal of the next step is to find an upper bound for
  \[ \eta_t \langle \nabla F(x_t^{md}),x_{t} - x_{t+1} \rangle - \frac{\mu}{q} \| x_{t+1} - x_{t} \|^q
  .\]

\subsection{Choice of aggregated point}

Recall from Algorithm~\ref{algo-acc} that 
\begin{align*}
x_{t+1}^{ag}  & = \arg\min_{x \in \mathbb{R}^d } \{ \alpha_t \langle \nabla F(x_t^{md}),x \rangle + \frac{\mu}{q} \| x - x_t^{md} \|^{q} \}
\end{align*}
and define $\phi_q(x) := \nabla \Big( \frac{\| x \|^q}{q} \Big)$. For all $u \in \mathbb{R}^d$, it holds
\[ \langle \alpha_t \nabla F(x_t^{md}) + \mu \phi_q(x_{t+1}^{ag} - x_t^{md}) , u - x_{t+1}^{ag} \rangle \geq 0.\]
Choosing $u = x_{t+1}^{ag} + \lambda (x_{t+1} - x_{t}) $ for any $\lambda > 0$, leads to
\begin{align*}
\lambda \alpha_t \langle  \nabla F(x_t^{md}) , x_{t} - x_{t+1} \rangle & \leq \lambda \mu \langle \phi_q(x_{t+1}^{ag} - x_t^{md}), x_{t+1} - x_t \rangle.
\end{align*}
Subtracting $\frac{\mu}{q} \| x_{t+1} - x_{t} \|^q$ on both sides, yields 
\begin{align}
\label{ageq}
 &\eta_t \langle  \nabla F(x_t^{md}) , x_{t} - x_{t+1} \rangle - \frac{\mu}{q} \| x_{t+1} - x_{t} \|^q \notag \\& \leq  \frac{\mu \eta_t}{\alpha_t} \langle  \phi_q(x_{t+1}^{ag} - x_t^{md}), x_{t+1} - x_t \rangle - \frac{\mu}{q} \| x_{t+1} - x_{t} \|^q\notag\\
  & {}={} \mu \Bigl \langle \frac{\eta_t}{\alpha_t}   \phi_q(x_{t+1}^{ag} - x_t^{md}), x_{t+1} - x_t \Bigr \rangle - \frac{\mu}{q} \| x_{t+1} - x_{t} \|^q
\end{align}
Let $q_\ast$ be the convex conjugate of $q$, i.e., \ $\frac{1}{q} + \frac{1}{q_{\ast}} = 1$.
A reformulation of Young's inequality in Lemma~\ref{young-ineq} gives
\begin{align*}
\langle a,b \rangle \leq  |\langle a,b \rangle | \leq \| a \|_{\ast} \cdot  \| b \| & \leq \frac{\| a \|_{\ast}^{q_\ast}}{{q_{\ast}}} + \frac{\| b \|^q }{q},
\end{align*}
leading to 
\begin{align*}
 \langle a,b \rangle - \frac{\| b \|^q }{q}  & \leq   \frac{\| a \|_{\ast}^{q_\ast}}{{q_{\ast}}}.
\end{align*}
Therefore \eqref{ageq} can be upper bounded by
\begin{equation}
\label{aggregated_mirror}
\begin{aligned}
 \mu \Bigl \langle \frac{\eta_t}{\alpha_t}   \phi_q(x_{t+1}^{ag} - x_t^{md}), x_{t+1} - x_t \Bigr \rangle - \frac{\mu}{q} \| x_{t+1} - x_{t} \|^q
 &\leq{}  \frac{\mu}{q_{\ast} } \Bigl \| \frac{\eta_t}{ \alpha_t}  \phi_q (x_{t+1}^{ag} - x_t^{md}) \Bigr \|_{\ast}^{q_{\ast}}  \\
&{}={}  \frac{\mu \eta_t^{q_{\ast}}}{q_{\ast}  \alpha_t^{q_{\ast}}} \| \phi_q(x_{t+1}^{ag} - x_t^{md}) \|_{\ast}^{q_{\ast}}  \\
&{}={}    \frac{ \mu  \eta_t^{q_{\ast}}}{q_{\ast} \alpha_t^{q_{\ast}}} \| x_{t+1}^{ag} - x_t^{md} \|^{(q-1 )q_{\ast}}\\
&{}={} \frac{ \mu  \eta_t^{q_{\ast}}}{q_{\ast}  \alpha_t^{q_{\ast}}} \| x_{t+1}^{ag} - x_t^{md} \|^{q},
\end{aligned}
\end{equation}
where the penultimate equality follows from Lemma~\ref{proof-regular-lemma}.
By setting $u = x_{t}^{md}$, we obtain from first-order optimality, that
\begin{align*}
\alpha_t  \langle \nabla F(x_t^{md}) , x_{t}^{md} - x_{t+1}^{ag} \rangle  & \geq \mu \langle \phi_q (x_{t+1}^{ag} - x_t^{md}) ,  x_{t+1}^{ag} - x_{t}^{md} \rangle,
\end{align*}
leading to
\begin{align*}
\alpha_t \Big( F(x_t^{md}) - F(x_{t+1}^{ag}) \Big) & \geq \mu \| x_{t}^{md} - x_{t+1}^{ag} \|^q - \alpha_t D_F(x_{t+1}^{ag}, x_t^{md}) \\
& \geq   \mu \| x_{t}^{md} - x_{t+1}^{ag} \|^q  - \frac{L\alpha_t}{\kappa}\| x_{t}^{md} - x_{t+1}^{ag} \|^\kappa, 
\end{align*}
where we use that $F$ is weakly smooth for the last inequality. By choosing
\begin{align*}
r & := \frac{q-\kappa}{\kappa} > 0 \\
M & :=   \Big( \frac{r}{q} \Big)^{r}  \geq 0,
\end{align*}
in Lemma~\ref{Lemma inexact gradient},  
it follows for all $\delta_t > 0$
\begin{align*}
  L\alpha_t \frac{\| x_{t}^{md} - x_{t+1}^{ag}  \|^\kappa}{\kappa} \leq  \frac{LM \alpha_t}{q \delta_t^r} \| x_{t}^{md} - x_{t+1}^{ag}  \|^q + L\alpha_t \delta_t.%
  \intertext{This leads to}\\
  \alpha_t \Big( F(x_t^{md}) - F(x_{t+1}^{ag}) \Big) \geq \Big( \mu - \frac{LM \alpha_t}{q \delta_t^r} \Big) \| x_{t}^{md} - x_{t+1}^{ag}  \|^q - L\alpha_t \delta_t
  .
\end{align*}
For convenience, we set $\frac{L M \alpha_t}{q \delta_t^r} = \frac{\mu}{q}$ by choosing $\delta_t = \Big( \frac{L}{\mu}  M\alpha_t \Big)^{1/r}$ and then get

\begin{align*}
\alpha_t \Big( F(x_t^{md}) - F(x_{t+1}^{ag}) \Big)  \geq \frac{\mu}{q_\ast} \| x_{t}^{md} - x_{t+1}^{ag}  \|^q - \frac{M^{1/r} }{\mu^{1/r}}\Big( L\alpha_t \Big)^{(r+1)/r}.
\end{align*}
Combining the previous equation with \eqref{ageq} and \eqref{aggregated_mirror}, yields
\begin{equation}\label{horror}
\begin{aligned}
& \eta_t \langle  \nabla F(x_t^{md}) , x_{t} - x_{t+1} \rangle  - \frac{\mu}{q} \| x_{t+1} - x_{t} \|^q  \\
\leq & \frac{\mu  \eta_t^{q_{\ast}}}{\alpha_t^{q_{\ast}}} \frac{\| x_{t+1}^{ag} - x_t^{md} \|^{q} }{q_\ast  }  \\
\leq & \frac{ \eta_t^{q_{\ast}}}{ \alpha_t^{q_{\ast}}} \Bigg( \alpha_t \Big( F(x_t^{md}) - F(x_{t+1}^{ag}) \Big)  +   \frac{M^{1/r} }{\mu^{1/r}}\Big( L\alpha_t \Big)^{(r+1)/r} \Bigg) \\
= & \frac{   \eta_t^{q_{\ast}}}{\alpha_t^{q_{\ast}-1}} \Big( F(x_t^{md}) - F(x_{t+1}^{ag}) \Big)  +   \frac{ \eta_t^{q_{\ast}}}{ \alpha_t^{q_{\ast}}}  \frac{M^{1/r} }{\mu^{1/r}} \Big( L\alpha_t \Big)^{(r+1)/r} \\
\end{aligned}
\end{equation}

\subsection{Simplified equation}

To simplify notation, we will abbreviate
\begin{equation}\label{notation-proof-p>2}
\begin{aligned}
A_t  & := \frac{  \eta_t^{q_{\ast}}}{ \alpha_t^{q_{\ast}-1}}   \\
B_t  & :=  \frac{  \eta_t^{q_{\ast}}}{ \alpha_t^{q_{\ast}}}  \frac{M^{1/r} }{\mu^{1/r}}  \Big( L\alpha_t \Big)^{1+1/r}  = \frac{ \eta_t^{q_{\ast}}}{\alpha_t^{q_{\ast}-1-1/r}} \frac{M^{1/r} }{\mu^{1/r}}   L ^{1+1/r}. \\ 
\end{aligned}
\end{equation}
Rearranging \eqref{mirror_analysis} and combining it with equation \eqref{horror} leads to
\begin{equation}
  \label{eq:blu}
\begin{aligned}
 &  A_t   \Big( F(x_{t+1}^{ag}) - F_\star \Big) \\
\leq  &  D_{\psi}(x_\star ,x_t) -  D_{\psi}(x_\star,x_{t+1}) + \eta_t \langle \nabla F(x_t^{md}),x_t^{md} - x_{t}  \rangle \\
 & + \Big(   A_t  -  \frac{\eta_t}{\tau} \Big) \Big( F(x_t^{md}) - F(x_t^{ag})  \Big)  \\
& + \Big(   A_t  -  \frac{\eta_t}{\tau} \Big) \Big( F(x_t^{ag}) - F_\star  \Big)  +   B_t
.
\end{aligned}
\end{equation}
With the binary search for $\lambda_t$, Algorithm~\ref{algo-acc} achieves
\begin{align*}
 \eta_t \langle \nabla F(x_t^{md}),x_t^{md} - x_{t}  \rangle 
+  \Big(  A_t  -  \frac{\eta_t}{\tau} \Big) \Big( F(x_t^{md}) - F(x_t^{ag})  \Big) \leq \eta_t \epsilon_t 
\end{align*}
within $\mathcal{O}(\log(\frac{L}{\epsilon_t}))$ iterations of function values and gradients. Thus, \eqref{eq:blu} becomes
\begin{equation}
\label{eqA}
\begin{aligned}
   A_t   \Big( F(x_{t+1}^{ag}) - F_\star \Big)
\leq &   D_{\psi}(x_\star ,x_t) -  D_{\psi}(x_\star,x_{t+1}) + \eta_t \epsilon_t 
\\
&
+  \Big(   A_t  -  \frac{\eta_t}{\tau} \Big) \Big( F(x_t^{ag}) - F_\star  \Big)  +   B_t. 
\end{aligned}
\end{equation}
To derive the optimal convergence rate, we need the following conditions to be fulfilled for a fixed $\alpha >0$:
\begin{align*}
\begin{cases}
A_t-\frac{\eta_t}{\tau} & \le A_{t-1}\\
B_t & \le  \alpha^{\frac{q}{q-\kappa}} \frac{1}{t}  \frac{M^{1/r} }{ \mu^{1/r}} L ^{1+1/r}.
\end{cases}
\end{align*}
If the previous conditions hold, we can compute the telescopic sum in \eqref{eqA} and get
\begin{equation}
\begin{aligned}
 &  A_T   \Big( F(x_{T+1}^{ag}) - F_\star \Big) \leq    D_{\psi}( x_\star ,x_1)   + \mathcal{O}\Big(   \alpha^{\frac{q}{q-\kappa}}    \frac{  L ^{1+1/r} }{ \mu^{1/r}}   \log (T) \Big).
\end{aligned}
\end{equation}
To find the optimal step-size, we choose
\begin{align*}
\beta  =  q-\kappa + \frac{q-\kappa}{q}, \quad 
\gamma  \geq 0, \quad \Big( \frac{\eta_t}{\alpha_t}\Big)^{q_\star - 1}  = \frac{1}{\tau(q-\beta)}  t^{\gamma(q_\star - 1)} 
\end{align*}
and recall that $q-\beta = \frac{q\kappa + \kappa - q}{q} > 0$, $1+ \frac{1}{r} = 1 + \frac{\kappa}{q-\kappa} = \frac{q}{q-\kappa}$ and $\frac{q_\ast}{q} = \frac{1}{q-1}$. Then it follows
\begin{align*}
   B_t \leq  \alpha^{\frac{q}{q-\kappa}} \frac{ 1}{t}  \frac{M^{1/r} }{ \mu^{1/r}} L ^{1+1/r},  
\end{align*}
leading to 
\begin{align*}
 \frac{   \eta_t^{q_{\ast}}}{\alpha_t^{q_{\ast}-1-1/r}}  \frac{M^{1/r} }{\mu^{1/r}} L ^{1+1/r} \leq  \alpha^{\frac{q}{q-\kappa}} \frac{ 1}{t}  \frac{M^{1/r} }{ \mu^{1/r}} L ^{1+1/r}
 \end{align*}
and thus 
 \begin{align*}
\alpha_t^{1+ 1/r}    \leq  \Big(\tau (q-\beta) \Big)^{\frac{q_\star}{q_\star - 1 }} \frac{ \alpha^{\frac{q}{q-\kappa}}}{t^{\gamma q_\ast + 1}}     \Leftrightarrow & \quad \alpha_t \leq   \Big(\tau (q-\beta) \Big)^{q-\kappa} \frac{\alpha}{t^{\frac{\gamma(q-\kappa)}{q-1}  + \frac{q-\kappa}{q}} }.
\end{align*}
We further note that 
\begin{align*}
A_t- A_{t-1}  =   \Big( \frac{\eta_t^{q_{\ast}}}{\alpha_t^{q_{\ast}-1}} - \frac{\eta_{t-1}^{q_{\ast}}}{\alpha_{t-1}^{q_{\ast}-1}} \Big) =   \left( \eta_t \Big( \frac{\eta_t}{\alpha_t}\Big)^{q_\star - 1} - \eta_{t-1} \Big( \frac{\eta_{t-1}}{\alpha_{t-1}}\Big)^{q_\star - 1} \right)   \leq \frac{ \eta_t}{\tau} 
\end{align*}
is equialent to
 \begin{align*}
\frac{A_t- A_{t-1} }{\eta_t} = \frac{1}{\tau(q-\beta)} \Big( t^{\gamma(q_\ast - 1)}   - \frac{\eta_{t-1}}{\eta_t} (t-1)^{\gamma(q_\ast - 1)}  \Big) \leq \frac{1}{\tau}.
\end{align*}
Intuitively, choosing $\eta_t$ as a polynomial will ensure that  $\frac{\eta_{t-1}}{\eta_t}$ tends to $1$.  Fixing the parameters in the following way satisfies our condition
\begin{align*}
\gamma & := q-1 \\
\alpha_t &:=     \Big(\tau (q-\beta) \Big)^{q-\kappa} \frac{\alpha}{t^\beta}   \\
\eta_t & := \alpha_t \Big(\frac{1}{\tau(q-\beta)}t^{\gamma(q_\star - 1)}  \Big)^{\frac{1}{q_\star - 1}} = \alpha_t \Big(\frac{t}{\tau(q-\beta)} \Big)^{q-1}.
\end{align*}
Note that $\gamma(q_\ast - 1) = 1 $.  We verify that
\begin{align*}
 \frac{\eta_{t-1}}{\eta_t} = & \frac{\alpha_{t-1}}{\alpha_t} \Big(1 - \frac{1}{t}\Big)^{q-1} = \Big(1 - \frac{1}{t}\Big)^{q-1-\beta} = \Big(\frac{t-1}{t}\Big)^{q -1 - \beta} 
 \end{align*}
 and
 \begin{align*}
\frac{A_{t} - A_{t-1}}{\eta_t} = &  \frac{1}{\tau(q-\beta)}  \Big( t   - \frac{\eta_{t-1}}{\eta_t} (t-1) \Big) \\
= &  \frac{1}{\tau(q-\beta)} \frac{1}{t^{q-1-\beta}} \Big( t^{q-\beta}   - (t-1)^{q-\beta} \Big) \\
= &  \frac{1}{\tau(q-\beta)} \frac{1}{t^{q-1-\beta}} \Big( \int_{t-1}^{t} (q-\beta) s^{q -\beta - 1} ds \Big)  \\
\leq &  \frac{1}{\tau} \frac{1}{t^{q-\beta-1}} \cdot t^{q-1-\beta} = \frac{1}{\tau}
\end{align*}
By choosing $\epsilon_t: = \frac{B_t}{\eta_t} $ such that $\eta_t \epsilon_t \leq B_t $, this leads to
\begin{align*}
A_t & = \frac{  \eta_t^{q_{\ast}}}{\alpha_t^{q_{\ast}-1}}   =    \Big( \frac{\eta_t}{\alpha_t} \Big)^{q_\ast} \alpha_t = \Big(\frac{t}{\tau(q-\beta)  }   \Big)^{\frac{q_\star}{q_\star - 1}}  \alpha_t  \\
& \propto  \Big( \frac{t}{  \tau} \Big)^{q} \cdot   \frac{\alpha}{t^{\beta}} =  \alpha t^{q-\beta}\cdot \frac{1}{\tau^q}  \\
A_t & \propto \frac{\alpha}{\tau^{q}} t^{\frac{\kappa q + \kappa - q}{q}  }
\end{align*}
and finally shows 
\begin{align*}
   F(x_{T+1}^{ag}) - F_\star  & = \mathcal{O}_{q,\kappa} \Big(   \frac{  \tau^q D_{\psi}( x_\star ,x_1) }{\alpha T^{\frac{\kappa q + \kappa - q}{q}  }}   +      \frac{ \alpha^{\frac{\kappa}{q-\kappa}} L ^{\frac{q}{q-\kappa}} }{ \mu^{\frac{\kappa}{q-\kappa}}} \frac{ \tau^q  \log (T)  }{T^{\frac{\kappa q + \kappa - q}{q}  } }  \Big),
\end{align*}
which proves Theorem~\ref{thm-p>2}.

\section{Proof for bounded iterate}
\begin{theorem}
Let's assume that $D_{\psi}(x_\star,x_1) \leq B $. Under assumption~\ref{ass}, we also suppose that we are running Algorithm~\ref{algo-acc} with parameters tuned such that there are base $K$ and exponents $n_1,n_2 \geq 0$ (possibly depending on $(L,\tau,\mu)$) ensuring $K^{n_1} > B$ and for all $t \geq 1$
\begin{align*}
  \max \left(  \Big(\frac{qL\eta_t}{\kappa} \Big)^{\frac{1}{q-1}} , \Big( \frac{L\alpha_t}{\kappa} \Big)^{\frac{1}{q-1}}   \right)\leq  K^{n_1}  t^{n_2}
  .
\end{align*}
If we write $R_t = \max\{ \| x_t - x_\star\|, \| x_t^{ag} - x_\star \|, \| x_t^{md} - x_\star \|\}$, then for all $t \geq 1$
\[ R_t = \mathcal{O} \Big( K^{\frac{n_1(q-1)}{q-\kappa}} t^{\frac{(q-1)(n_2 + 1)}{q-\kappa}} \Big)\]
\end{theorem}

\begin{proof}
We recall from our Algorithm~\ref{algo-acc} that
\begin{align*}
x_{t+1}  & = \arg \min_{x \in \mathbb{R}^d} \{ \eta_t \langle \nabla F(x_{t}^{md}), x  \rangle +  D_\psi(x,x_t) \} \\ x_{t+1}^{ag} &  = \arg \min_{x \in \mathbb{R}^d}  \{ \alpha_t \langle \nabla F(x_{t}^{md}), x  \rangle +  \frac{\| x - x_t^{md} \|^q}{q}  \}    
\end{align*}
Since our setting is unconstrained, the (sub-)gradients have to cancel :
\begin{align*}
\eta_t \nabla F(x_t^{md}) + \nabla \psi(x_{t+1}) - \nabla \psi(x_t) = 0 \\
\alpha_t \nabla F(x_t^{md}) + \phi_q (x_{t+1}^{ag} - x_t^{md}) = 0 
\end{align*}
with $\psi(x) \propto \frac{\| x\|^q}{q}$ and $\phi_q(x) = \nabla \Big( \frac{\|x \|^q}{q}\Big)$.
By the weak smoothness property of $F$:
\begin{align*}
\| \nabla F(x_t^{md}) \|_{\ast} \leq \frac{L}{\kappa}\|x_t^{md} - x_{\star} \|^{\kappa-1}.
\end{align*}
Additionally, from the uniform convexity of $\psi$ and regularity of $\| \cdot \|$, it follows that:
\begin{align*}
\frac{\| x_{t+1} - x_t \|^q}{q} \leq D_{\psi}(x_t, x_{t+1}) + D_{\psi}(x_{t+1}, x_{t})& = \langle \nabla \psi(x_{t+1}) - \nabla \psi(x_t),x_{t+1}- x_t \rangle.
\end{align*}
Applying the optimality condition for $x_{t+1}$, this becomes
\begin{align*}
    \frac{\| x_{t+1} - x_t \|^q}{q} \leq \eta_t \langle \nabla F(x_t^{md}), x_t - x_{t+1} \rangle
\end{align*}
and by applying Cauchy-Schwartz, we obtain:
\begin{align*}
   \frac{\| x_{t+1} - x_t \|^q}{q} \leq \eta_t   \| \nabla F(x_t^{md}) \|_{\ast} \times \|  x_t - x_{t+1} \|, 
\end{align*}
where $\|\cdot\|_\ast$ is the dual norm associated with $\|\cdot\|$.
Similarly for $x_{t+1}^{ag}-x_t^{md}$, we use the definition of $\phi_q$ and obtain from Lemma~\ref{proof-regular-lemma} that

\begin{align*}
\| x_{t+1}^{ag} - x_{t}^{md} \|^q & =  \langle \phi_q(x_{t+1}^{ag} - x_{t}^{md}) , x_{t+1}^{ag} - x_{t}^{md} \rangle \\
& \leq \alpha_t   \| \nabla F(x_t^{md}) \|_{\ast} \times \|  x_{t+1}^{ag} - x_{t}^{md} \|.
\end{align*}
By plugging in these results in the previous conditions, we derive the following bounds:
\begin{align*}
\|x_{t+1} - x_t \|^{q-1}  \leq \frac{qL\eta_t}{\kappa}\|x_t^{md} - x_{\star} \|^{\kappa-1} \\
\| x_{t+1}^{ag} - x_{t}^{md} \|^{q-1}  \leq \frac{L\alpha_t}{\kappa}\|x_t^{md} - x_{\star} \|^{\kappa-1}.
\end{align*}
From the triangle inequality, we then obtain :
\begin{align*}
\| x_{t+1} - x_\star \| & \leq \|x_t - x_\star\| + \| x_{t+1} - x_t \| \\
& \leq \|x_t - x_\star\| + \Big(  \frac{qL\eta_t}{\kappa} \Big)^{\frac{1}{q-1}} \|x_t^{md} - x_{\star} \|^{\frac{\kappa-1}{q-1}}
\end{align*}
and 
\begin{align*}
\| x_{t+1}^{ag} - x_\star \| & \leq \|x_t^{md} - x_\star\| + \Big(  \frac{L\alpha_t}{\kappa} \Big)^{\frac{1}{q-1}} \|x_t^{md} - x_{\star} \|^{\frac{\kappa-1}{q-1}}.
\end{align*}
Since $x_{t+1}^{md}$ is a middle point, we can also deduce that 
\begin{align*}
\| x_{t+1}^{md} - x_\star  \| & = \| \lambda_{t+1} (x_{t+1}^{ag} - x_\star) + (1- \lambda_{t+1} )(x_{t+1} - x_\star) \| \\
& \leq \lambda_{t+1 }\| x_{t+1}^{ag} - x_\star\| + (1- \lambda_{t+1} )\|x_{t+1} - x_\star \| \\
& \leq \max \Big( \| x_{t+1}^{ag} - x_\star\|, \|x_{t+1} - x_\star \| \Big) 
\end{align*}
From the above inequality and using the assumption on $(\alpha_t,\eta_t)$ it follows that:
\[ R_{t+1} \leq R_t +  K^{n_1} t^{n_2}  R_t^{\frac{\kappa - 1}{q-1}}.\]
This recursive inequality shows that the growth is at most polynomial, implying that the sequence $R_t$ cannot grow faster than a polynomial rate.

\bigskip\noindent
We will show by induction that $R_t \leq K^{S_t}  t^{M_2} $, where $M_2 := \max\Big\{ \frac{(q-1)(n_2 + 1)}{q-\kappa},1\Big\}$ and $S_t := n_1 \sum_{i=0}^{t-1} \Big( \frac{\kappa -1}{q-1} \Big)^i $ is a geometric sum. First, we notice that
\begin{align*}
 n_2 + 1 & \leq  \frac{q-\kappa}{q-1} M_2 \\
\implies   n_2 + M_2 \frac{\kappa - 1}{q-1} & \leq M_2 - 1 \\
\implies  t^{n_2 + M_2 \frac{\kappa - 1}{q-1}} & \leq t^{M_2 - 1} \leq (t+1)^{M_2} - t^{M_2}.
\end{align*}
Then, by induction it then follows that
\begin{align*}
R_{t+1} \leq R_t +  K^{n_1} t^{n_2}  R_t^{\frac{\kappa - 1}{q-1}} & \leq  K^{S_t}  t^{M_2} +  K^{n_1 + S_t \times \frac{\kappa-1}{q-1}} t^{n_2 + M_2 \frac{\kappa - 1}{q-1}}  \\
& \leq K^{S_{t+1}}  \Big( t^{M_2 } + t^{n_1 + M_2 \frac{\kappa - 1}{q-1}} \Big) \\
& \leq K^{S_{t+1}} (t+1)^{M_2}
\end{align*}
Since $S_t < \frac{n_1}{1 - \frac{\kappa- 1}{q-1}} = \frac{n_1(q-1)}{q-\kappa}$ , we can simplify $R_t = \mathcal{O} \Big( K^{\frac{n_1(q-1)}{q-\kappa}} t^{\frac{(q-1)(n_2 + 1)}{q-\kappa}} \Big)$.
\end{proof}

\section{Proof for binary search}
Before proving Theorem~\ref{thm-binary-search}, we make two quick observations.
\begin{lemma}
\label{lemma-binary-search}
Let $F : \mathbb{R}^d \xrightarrow{} \mathbb{R}$ be a $(L, \kappa)$-weakly smooth function. Define $g:\mathbb{R} \to \mathbb{R}$ for $a,b \in \mathbb{R}^d$ by
\[g(\lambda) = F(a + \lambda b ). \]
Then for any $\lambda_1, \lambda_2$, we have 
\begin{align*}
|g'(\lambda_1) - g'(\lambda_2)| \leq \frac{2L}{\kappa} \|b \|^\kappa |\lambda_1 - \lambda_2|^{\kappa}.
\end{align*}
\end{lemma}

\begin{proof}
We recall that :
\[ g'(\lambda) = \langle \nabla F(a+\lambda b), b \rangle \]
and for any $\lambda_1, \lambda_2$, by weak smoothness:
\begin{align*}
- \frac{2 L}{\kappa} \| \lambda_1 b - \lambda_2 b \|^\kappa \leq D_F(a+\lambda_1 b ,a+\lambda_2 b  ) + D_F (a+\lambda_2 b ,a+\lambda_1 b) ) \leq \frac{2L}{\kappa} \| \lambda_1 b - \lambda_2 b \|^\kappa \\
- \frac{2L}{\kappa} \| \lambda_1 b - \lambda_2 b \|^\kappa \leq  \langle \nabla F( a + \lambda_1 b) - \nabla F( a + \lambda_2 b),b \rangle \leq \frac{2L}{\kappa} \| \lambda_1 b - \lambda_2 b \|^\kappa 
\end{align*}
By taking the absolute value :
\begin{align*}
  |g'(\lambda_1) - g'(\lambda_2)| \leq \frac{2L}{\kappa} \| \lambda_1 b - \lambda_2 b \|^\kappa = \frac{2L}{\kappa} \|b \|^\kappa |\lambda_1 - \lambda_2|^\kappa.
\end{align*}
\end{proof}

We notice that, since we use the binary search in the segment $[0,1]$ (i.e. $g:[0,1] \xrightarrow{} \mathbb{R}$), $F$ being $(L,\kappa)$-weakly smooth implies that $g'$ is $\frac{2L}{\kappa} \|x_t - x_t^{ag} \|^\kappa $-Lipschitz.

\bigskip\noindent
\begin{lemma}
\label{lemma-binary-search-2}
Consider $a<b$ and let $g :\mathbb{R} \xrightarrow{} \mathbb{R} $ have a continuous derivative. Then
\[ g(a),g'(b) > 0 \geq g(b) ~\implies~ \exists \lambda_\star \in [a,b], \quad g(\lambda_\star) \leq 0, g'(\lambda_\star)=0 \]
\end{lemma}

\begin{proof}
Let us pick $\lambda_\star = \sup \{\lambda \in [a,b] \mid g'(\lambda) \leq 0 \}$. That search set is not empty because due to Taylor-Lagrange, as $g'$ is continuous, there exists $\lambda \in [a,b]$ such that $g'(\lambda) = \frac{g(b) - g(a)}{b-a} < 0$. By continuity of $g$ and $g'$, we conclude that $g(\lambda_\star) \leq 0, g'(\lambda_\star) = 0$.
\end{proof}
Now we are ready to prove Theorem~\ref{thm-binary-search} on the complexity of the binary search.
\begin{theorem}
For fixed $(C_t,\epsilon_t,x_t,x_t^{ag})$, the binary search Algorithm~\ref{algo-binary-search} finds a $\lambda$ such that
\begin{align}
\label{cond}
x_{t}^{md} & = \lambda x_{t}^{ag} + (1-\lambda) x_{t} \notag\\
 \langle \nabla F(x_t^{md}),x_t^{md} - x_{t}  \rangle 
& +  C_t  \Big( F(x_t^{md}) - F(x_t^{ag})  \Big) \leq \epsilon_t  
\end{align}
in $\mathcal{O}_{\kappa}\left(\max\{ \log(C_t),  \log \Big( \frac{ L\|x_t - x_t^ {ag} \|^\kappa }{ \epsilon_t }\Big) \} + 1 \right) $ iterations.
\end{theorem}

\begin{proof} \label{appendix-binary-search-proof}
Set
\[ g(\lambda) := F(x_t^{md}) - F(x_t^{ag}) = F(\lambda x_t^{ag} + (1-\lambda)x_t ) - F(x_t^{ag}).\]
We notice that $g(1) = 0$ and 
\begin{align*}
g'(\lambda) & = \langle \nabla F(x_t^{md}), x_t^{ag} -x_t \rangle \\
\langle \nabla F(x_t^{md}),x_t^{md} - x_{t}  \rangle & = \langle \nabla F(x_t^{md}), \lambda (x_t^{ag} - x_{t})  \rangle = \lambda g'(\lambda).
\end{align*}
As a first step in our proof, we show that there indeed exists a $\lambda$ satisfying \eqref{cond}, i.e., fulfilling
\[ C_t g(\lambda)  + \lambda g'(\lambda )\leq \epsilon_t. \]
If $\lambda = 0$ and $\lambda = 1$ fail to satisfy the inequality, we know that both :
\begin{align*}
g'(1) & > \epsilon_t  \\
C_t g(0)  & > \epsilon_t 
\end{align*}
From now, we will assume $g'(1),g(0) > 0$ for the rest of the proof. We set an index set $\Lambda_\star = \{ \lambda \in [0,1] \mid g(\lambda) \leq 0, g'(\lambda)=0 \}$. We notice that
\[\lambda \in \Lambda_\star \implies C_t g(\lambda)  + \lambda g'(\lambda )\leq 0.\]
We show that $\Lambda_\star$ is not an empty set. For that we will apply Lemma~\ref{lemma-binary-search-2} with $g(0),g'(1) > 0 = g(1)$, there exists $\lambda_0 \in [0,1], g(\lambda_0) \leq 0, g'(\lambda_0) = 0 $ and $\lambda_0 \in \Lambda_\star$.

\bigskip\noindent
Now for any $\lambda_\star \in \Lambda_\star$, since $g'$ is $\alpha$-Holder-smooth according to Lemma~\ref{lemma-binary-search}, for $\delta > 0$,
\begin{align*}
\forall \lambda \in [\lambda_\star-\delta,\lambda_\star]: \quad | g'(\lambda) - g'(\lambda_\star) | \leq \frac{2L}{\kappa} \|x_t - x_t^ {ag} \|^\kappa \delta^\kappa \\
\forall \lambda \in [\lambda_\star-\delta,\lambda_\star]: \quad g(\lambda) - g(\lambda_\star) \leq \int_{\lambda_\star-\delta}^{\lambda_\star}  |g'(t)| dt   \leq \frac{2L}{\kappa} \|x_t - x_t^ {ag} \|^\kappa \delta^{\kappa+1}
\end{align*}
Therefore, for all $ \lambda \in [\lambda_\star-\delta,\lambda_\star]$,
\begin{align*}
  &C_t g(\lambda)  + \lambda g'(\lambda )
  \\
  &~=~
  \underbrace{
    C_tg(\lambda_\star)  + \lambda_\star g'(\lambda_\star )
    }_{\le 0}
    + C_t(g(\lambda) - g(\lambda_\star))
    + \lambda g'(\lambda )
    - \lambda_\star \underbrace{g'(\lambda_\star )}_{=0}
    \\
    &~\le~
    C_t \frac{2L}{\kappa} \|x_t - x_t^ {ag} \|^\kappa \delta^{\kappa+1}
    + \lambda \frac{2L}{\kappa} \|x_t - x_t^ {ag} \|^\kappa \delta^\kappa
    \\
    &~\le~
    \frac{2L}{\kappa} \|x_t - x_t^ {ag} \|^\kappa \delta^\kappa (
    C_t\delta
    + 1)
    .
\end{align*}
In order to make the last expression smaller than $\epsilon_t$, we consider $\delta = \min \{\frac{1}{C_t},\Big( \frac{\kappa \epsilon_t }{4 L\|x_t - x_t^ {ag} \|^\kappa }\Big)^{1/\kappa} \}$.
The last step is to notice that in Algorithm~\ref{algo-binary-search}, our stop condition is satisfied when 
\[\Lambda_\star \cap [a,b] \neq \emptyset , \delta > b-a,\]
since that implies $\exists \lambda_\star \in \Lambda_\star, a \in [\lambda_\star - \delta, \lambda_\star] \implies C_tg(a) + a g(a) \leq \epsilon_t$.

Now we verify that the condition $\Lambda_\star \in [a,b]$ holds for all the iterates. For that, using the previous Lemma~\ref{lemma-binary-search-2}, we only need to verify the conditions at the extremities of the segment.
If we initially have :
\[g(a), g'(b) > 0 \geq g(b)\]
then while iterating with $\frac{a+b}{2} $, we know that :
\begin{align*}
g \Big( \frac{a+b}{2} \Big) > 0  \implies g \Big( \frac{a+b}{2} \Big), g'(b) > 0 \geq g(b) & \implies \Lambda_\star \cap \Big[\frac{a+b}{2} , b\Big]  \neq \emptyset \\
g'\Big( \frac{a+b}{2} \Big) > 0 \geq g \Big( \frac{a+b}{2} \Big)   \implies g(a), g'\Big( \frac{a+b}{2} \Big) > 0 \geq g\Big( \frac{a+b}{2} \Big) & \implies \Lambda_\star \cap \Big[a, \frac{a+b}{2}\Big]  \neq \emptyset \\
0 \geq g \Big( \frac{a+b}{2} \Big), g'\Big( \frac{a+b}{2} \Big)    & \implies \frac{a+b}{2}  \in \Lambda_\star
\end{align*}
Once we prove that our iteration is valid, we know the interval starts as $[0,1]$ and we halve the size in each iteration until it reaches $\delta$. So the maximum iteration count is  :
\begin{align*}
\log(1/\delta)+1 = \max\left\{ \log(C_t),  \log \left( \frac{4 L\|x_t - x_t^ {ag} \|^\kappa }{\kappa \epsilon_t }\right)^{1/\kappa} \right\} + 1 
\end{align*}
\end{proof}

\section{Lower bounds for binary search}
We recall from Section~\ref{sec4} the sub-problem~\eqref{bin_cond} of finding $\lambda \in [0,1]$ such that
\begin{equation}\label{eq:g.eps.opt}
  C g(\lambda)  + \lambda g'(\lambda )\leq \epsilon
\end{equation}
with $C,\epsilon>0$ two constants and $g(1)=0$.
In this section, we show a lower bound on the number of iterations required to solve that problem in a first-order query model. We denote the class of weakly-smooth functions by
\[ \mathcal{F}_{L_\star,\kappa} := \{ g \in \mathcal{C}^1([0,1], \mathbb{R}) \mid   g(1)=0, \quad \forall a,b \in [0,1], \quad  | g'(a) - g'(b)| \leq L_\star|a-b|^{\kappa-1} \}, \]
and we denote the class of smooth functions by $\mathcal{F}_{L_\star} :=\mathcal{F}_{L_\star,2}$. Since the domain is $[0,1]$, we know that $\mathcal{F}_{L_\star} \subset \mathcal{F}_{L_\star,\kappa}$ for all $1<\kappa \leq 2$.
The lower bound consists of building two smooth functions that agree with the observed function values and derivatives $(g(\lambda_t),g'(\lambda_t))$ for every $\lambda_t$ evaluated, but these two functions have distinct intervals where the condition~\eqref{eq:g.eps.opt} is satisfied on the line. For simplicity, we refer to the condition ~\eqref{eq:g.eps.opt} as the linear condition.

The following theorem gives the key condition on the interval $[a,b]$ to build counter examples. More precisely, we build $g$ in the interval $[a,b]$ and we need to satisfy the interpolation condition $(g(a),g'(a))=(A,0)$ and $(g(b),g'(b))=(B,B')$. $C$ is the constant given in the problem \eqref{eq:g.eps.opt} and $C_\star$ is related to $C$ where we need $[a,b] \subset [1/C_\star,1]$ so that our lower bound will work. For the upcoming theorem, we will consider the following notation   
\begin{align}
  \notag
  C_\star & = 1 + \frac{1}{C}\\
  \label{eq:PHI}
\Phi(a,b,A,B,B') & =   \frac{56 B'}{L_\star(b-a)} + \frac{32(A-B)}{L_\star(b-a)^2 }
.
\end{align}

Now, we can give the main theorem for lower bound.

\begin{theorem} \label{lower-bound-binary}
If the following conditions are satisfied,
\begin{equation}
\begin{aligned}
\begin{cases}
[a,b]  \subset [ 1/C_\star,1] \\
\Phi(a,b,A,B,B') \leq  1     \\
A \geq \frac{\epsilon}{C}, \quad \quad C B+ \frac{B'}{C_\star} \geq \epsilon  \\
B' \geq 0 \geq B \\
\end{cases}
\end{aligned}
\end{equation}

then there exists two functions $g_1, g_2 \in \mathcal{F}_{L_\star} $ such that : 
\begin{align*}
g_1(a) = g_2(a) & = A\\
g_1(b) = g_2(b) & = B\\
g_{1}'(a) = g_{2}'(a) & = 0\\
g_{1}'(b) = g_{2}'(b) & = B'\\
\forall \lambda \in \Big[a, \frac{a+b}{2}\Big], \quad & Cg_1(\lambda) + \lambda g_1'(\lambda) \geq \epsilon \\
\forall \lambda \in \Big[\frac{a+b}{2},b \Big], \quad & Cg_2(\lambda) + \lambda g_2'(\lambda) \geq \epsilon
\end{align*}
\end{theorem}

We also notice that if the conditions are satisfied with $2\epsilon$ instead of $\epsilon$, we can have a strict guarantee on the last lines on $\epsilon$, therefore the strictness of inequality is not important here. Theorem~\ref{lower-bound-binary} claims that if the condition evaluation parameter $\Phi$ is smaller than $1$, it is impossible to determine the location of the correct $\lambda_t$ since there are still two functions in $\mathcal{F}_{L_\star}$ providing distinct intervals of solutions $\lambda_t$ (by the continuity of $g'$, a solution $\lambda_t$ has to exist over the whole interval $[a,b]$ if the interpolation conditions on $(a,b)$ are satisfied.) We provide the proof in the next section.

\subsection{Proof with two separate constructions}
In this section, we will show how to build the two smooth functions $(g_1,g_2)$ in two different ways in theory. (In practice, one only need to build them at the end of all requests to show the lower bounds.)  Both functions will interpolate (i.e.\ agree with) given first-order boundary conditions, but the first function $g_1$ has the property to not have any point satisfying \eqref{eq:g.eps.opt} on the left side $\Big[a,\frac{a+b}{2}\Big]$ whereas the second function $g_2$ does not satisfy the condition \eqref{eq:g.eps.opt} on the right side $\Big[\frac{a+b}{2},b\Big]$. We use two separate constructions because \eqref{eq:g.eps.opt} is not symmetric. The rough idea is that if the size of the interval $[a,b]$ is big enough, then the construction is possible while respecting the $L_\star$-smoothness condition.
\begin{lemma} \label{lemma-construction-1}
We consider $[a,b] \subseteq [0,1] $. If the following conditions are satisfied :
\begin{equation}
\begin{aligned}
\begin{cases}
b-a & \geq \frac{6 B'}{L_\star} +  \frac{16(A - B)}{ L_\star (b-a)}     \\
A & \geq \frac{\epsilon}{C} \\
B' & \geq 0 
\end{cases}
\end{aligned}
\end{equation}
then there exists an interpolating $L_\star$-smooth function $g_1$ such that :
\begin{align*}
g_{1}'(a) & = 0, \quad g_{1}(a) = A,  \\
g_{1}'(b) & = B', \quad g_1(b) = B, \quad   \\
\forall  x \in & \Big[a, a + \frac{b-a}{2} \Big], \quad  C g_1(\lambda)  + \lambda g_1'(\lambda)  \geq \epsilon 
\end{align*}
\end{lemma} 

\begin{proof}
We will make a three-piece piece-wise linear interpolation of the derivative $g'$:
\begin{align*}
g_{1}'\Big(a + \frac{b-a}{2} \Big) & = 0 \\
g_{1}'\Big(a + \frac{3(b-a)}{4} \Big) & =  - W
\end{align*}
with $W > 0$ that we determine later. Now we will verify the linear condition \eqref{eq:g.eps.opt} on different segments of $[a,b]$.
For $\lambda \in \Big[a, a + \frac{b-a}{2} \Big] $,  we have $g_{1}'(\lambda) = 0, \quad g_{1}(\lambda) = A $. Therefore, in this interval, we can always assure that \[C g_{1}(\lambda) + \lambda g_{1}'(\lambda) \geq CA \geq \epsilon. \]
For $\lambda \in \Big[a + \frac{b-a}{2}, a + \frac{3(b-a)}{4} \Big] $,  we have the linear interpolation 
\[ g_{1}'(\lambda) = - \frac{4 W}{b-a} \times \Big (\lambda - a - \frac{b-a}{2} \Big), \]
For $\lambda \in \Big[a + \frac{3(b-a)}{4}, b \Big] $,  we have the linear interpolation
\[ g_{1}'(\lambda) = -W  + \frac{4(B'+W)}{b-a} \times \Big (\lambda - a - \frac{3(b-a)}{4} \Big), \]
In the end we want to make sure that the integral of $g'$ over $[a,b]$ is $B-A$ :
\[ g_{1}\Big( a + \frac{b-a}{2} \Big)  + \int_{a + \frac{b-a}{2}}^{b} g_{1}'(\lambda) d \lambda = B.  \]
We can develop the previous form, since $g'$ is a linear interpolation, the integral is just the average of its two extreme points  :
\begin{align*}
 A - \frac{W(b-a)}{8} + \frac{b-a}{4} \times \frac{-W+B'}{2} & = B \\
\Leftrightarrow   -W -W + B' & = \frac{8(B - A)}{b-a}   \\
\Leftrightarrow   W & =  \frac{B'}{2} + \frac{4(A - B)}{b-a} 
\end{align*}
Now that we find the value for $W$, we need to check the smoothness on the two last intervals.
\begin{align*}
\begin{cases}
W & \leq L_\star \times  \frac{b-a}{4} \\
B' + W & \leq L_\star \times  \frac{b-a}{4}
\end{cases}
\end{align*}
Since we supposed that $B' \geq 0$, the second condition is stricter. We rewrite the second condition as : 
\begin{align*}
 \frac{B'}{2} + \frac{4(A - B)}{b-a}  +   B'  & \leq L_\star \times  \frac{b-a}{4}  \\
\Leftrightarrow    \frac{6 B'}{L_\star} +  \frac{16(A - B)}{ L_\star (b-a)}   & \leq (b-a)   \\
\end{align*}
which is what we set out to show.
\end{proof}

Now we will build the second interpolation function on the right side of the interval $\Big[ \frac{a+b}{2},b\Big]$. Unlike $g_1$, constant derivative is not working here, we show in the following lemma that linear derivative works if $[a,b]$ is big enough, but contained in $[ 1/C_\star,1]$.
\begin{lemma} \label{lemma-construction-2}
Let's consider $C_\star =1 + \frac{1}{C}$. We consider $[a,b] \subset [ 1/C_\star,1]$. If the following conditions are satisfied :

\begin{equation}
\begin{aligned}
\begin{cases}
b-a & \geq \frac{56 B'}{L_\star} + \frac{32(A-B)}{L_\star(b-a)}  \\
B' & \geq 0 \geq B \\
CB+ \frac{B'}{C_\star} & \geq \epsilon ,
\end{cases}
\end{aligned}
\end{equation}
then there exists an $L_\star$-smooth function $g_2$ such that :
\begin{align*}
g_{2}'(a) & = 0, \quad g_{2}(a) = A\\
g_{2}'(b) & = B', \quad g_1(b) = B \\
\forall  \lambda \in & \Big[a + \frac{b-a}{2}, b  \Big], \quad  Cg_2(\lambda)  + \lambda g_2'(\lambda)  \geq \epsilon 
\end{align*}
\end{lemma}

\begin{proof}
Now we start making linear interpolation for the derivatives with the middle point values fixed in the following way
\begin{align*}
g_{2}' \Big(a + \frac{b-a}{4} \Big) & = - W \\
g_{2}' \Big(a + \frac{b-a}{2} \Big) & =  KB'
\end{align*}

with $V,W > 0$ two constants to be determine later. Now we will start with the right side as it is the most interesting one.

For $\lambda \in \Big[a + \frac{(b-a)}{2},b \Big] $,  we apply the linear interpolation :
\begin{align*}
g_{2}'(\lambda) & =  B' + \frac{2(KB' -B')}{b-a} \times \Big ( b - \lambda \Big) \\
& =  B' \left( 1  + \frac{2(K-1)}{b-a} \times \Big ( b - \lambda \Big)   \right) 
\end{align*}

Since the interpolation of the derivative is linear, we know that : 
\begin{align*}
g_2(b) - g_{2}(\lambda) & = \int_{\lambda}^{b} g'(s) ds \\
& = \Big( \frac{g_{2}'(b) + g_{2}'(\lambda)}{2}\Big) (b-\lambda) \\
g_2(\lambda) & = B - B' \left(1  + \frac{(K-1)(b - \lambda)}{b-a}    \right) (b-\lambda)
\end{align*}

Since $a\geq 1/3$ and $g_2' \geq 0$ on this interval, we know that:
\begin{align*}
& Cg_2(\lambda) + x g_{2}'(\lambda) \\
\geq & Cg_2(\lambda) + \frac{ g_{2}'(\lambda) }{C_\star}\\
\geq & CB - CB' \left(1  + \frac{(K-1)(b - \lambda)}{b-a}    \right) (b-\lambda)  + \frac{B'}{C_\star}  \left( 1 +  \frac{2(K-1)}{(b-a)} \times ( b - \lambda) \right)   \\
\geq & \Big( CB +  \frac{B'}{C_\star}\Big)  +  \left(   \frac{2(K-1)}{C_\star(b-a)}  -  C \Big(1  + \frac{(K-1)(b - \lambda)}{b-a}    \Big)    \right) \times  B' ( b - \lambda)  
\end{align*}

We only need to choose $K$ big enough so $\frac{2(K-1)}{C_\star(b-a)}  -  C \Big(1  + \frac{(K-1)(b - \lambda)}{b-a}    \Big)   $ is always positive, we notice that :
\begin{align*}
\frac{b-\lambda}{b-a} & \leq \frac{1}{2} \\
b-a & \leq 1 - \frac{1}{C_\star} = \frac{C_\star - 1}{C_\star} \\
\implies \frac{1}{C_{\star}(b-a)} & \geq \frac{1}{C_\star} \frac{C_\star }{C_\star - 1} = \frac{1}{C_\star - 1}
\end{align*}
then 
\begin{align*}
& \frac{2(K-1)}{C_\star (b-a)}  -  C \Big(1  + \frac{(K-1)(b - \lambda)}{b-a}    \Big) \\
\geq &\frac{2(K-1)}{C_\star - 1} - C\Big(1  + \frac{K-1}{2}    \Big) \\
= & \frac{2(K-1)}{C_\star - 1} -  \frac{C(K+1)}{2}  \\
= & \frac{4}{C_\star - 1} - 2C \geq 0
\end{align*}

by picking $K=3$ in the last line and as we recall that $C_\star = 1 + \frac{1}{C}$.

For $x \in \Big[a, a + \frac{(b-a)}{2}\Big] $,  we need to assure the integral of $g'$ over $[a,b]$
\begin{align*}
\int_{a}^{a + \frac{b-a}{2}} g_{2}'(\lambda) d\lambda  =  g_2\Big(\frac{a+b}{2}\Big) - g_2(a) & = B - B' \left(1  + \frac{K-1}{2}    \right) \frac{b-a}{2} - A  \\
& =B - B'(b-a) - A 
\end{align*}

We are making an linear interpolation of $g_{2}'$ and 
\[ g_{2}'(a)= 0 , \quad g_{2}'\Big(a + \frac{b-a}{4} \Big) = - W, \quad g_{2}'\Big(a + \frac{b-a}{2} \Big) = 3B',\]

then,
\begin{align*}
\frac{0 + 3B' - 2W }{4} \times \frac{b-a}{2} & = B-A - B'(b-a)  \\
\Leftrightarrow 3B' - 2W + 8 B'  & = \frac{8(B-A)}{b-a} \\
W     & = 11B' +  \frac{8(A-B)}{b-a} 
\end{align*}

The remaining task is to check the $L_\star$-smoothness over the three segments :
\begin{align*}
\begin{cases}
W & \leq L_\star \times  \frac{b-a}{4} \\
W + 3 B' & \leq L_\star \times  \frac{b-a}{4} \\
2B' & \leq L_\star \times  \frac{b-a}{2}
\end{cases}
\end{align*}

The second one is the most restrictive one. It is equiavalent to :
\begin{align*}
11B' +  \frac{8(A-B)}{b-a}  + 3B' & \leq L_\star \times  \frac{b-a}{4}  \\
\Leftrightarrow  \frac{56 B'}{L_\star} + \frac{32(A-B)}{L_\star(b-a)}  & \leq b-a  
\end{align*}

\end{proof}

\subsection{Induced Lower bound theory}

Now combining Lemma~\ref{lemma-construction-1} with Lemma~\ref{lemma-construction-2}, we will design a lower bound adversary $\mathcal{A}$ that limits the growth of $\Phi$ from \eqref{eq:PHI} at each evaluation.

If $\lambda \in \Big[ a, a + \frac{b-a}{2} \Big] $, $\mathcal{A}$ returns :
\[ g'(\lambda) = \epsilon, \quad g(\lambda) = A \]

If $\lambda \in \Big[  a + \frac{b-a}{2}, b \Big] $, $\mathcal{A}$ returns :
\[ g'(\lambda) = B' \left( 1  + \frac{4(b-\lambda)}{b-a}   \right) , \quad g(\lambda) = B - B' \left(1  + \frac{2(b - \lambda)}{b-a}    \right) (b-\lambda) \]

\begin{theorem} 
We consider the same condition as in Theorem~\ref{lower-bound-binary}. After requesting $(g,g')$ at $\lambda$, the algorithm $\mathcal{A}$ returns $(g(\lambda),g'(\lambda))$. We know that value of $\Phi$ on  $[a', b'] = [a,\lambda] $ or  $[a', b'] = [\lambda,b] $ will verify that :
\begin{align*}
\begin{cases}
\Phi\Big(a',b',g(a'),g(b'),g'(b') \Big)  \leq 5  \Phi(a,b,A,B,B')    \\
g(a')\geq \frac{\epsilon}{C},  \quad g' (b') \geq 0 \geq g(b')\\
C g(b') +  \frac{g'(b')}{C_\star} \geq \epsilon , \\
\end{cases}
\end{align*}
\end{theorem}

\begin{proof}
If $\lambda \in \Big[ a, a + \frac{b-a}{2} \Big] $, the new segment is on the right side of $\lambda$, $[a', b'] = [\lambda,b] $. Since $b-a \leq 2(b-\lambda)$, we know that
\begin{align*}
\Phi\Big(a',b',g(a'),g(b'),g'(b') \Big) & = \frac{56 g'(b')}{L_\star(b-\lambda)} + \frac{32(g(\lambda)-g(b))}{L_\star(b-\lambda)^2 }  \\
& = \frac{56 B'}{L_\star(b-\lambda)} + \frac{32(A-B)}{L_\star(b-\lambda)^2 } \\
& \leq 2\times  \frac{56 B'}{L_\star(b-a)} + 4\times  \frac{32(A-B)}{L_\star(b-a)^2 } \leq 4 \Phi(a,b,A,B,B')
\end{align*}
If $\lambda \in \Big[  a + \frac{b-a}{2}, b \Big] $, the new segment is on the left side of $\lambda$, $[a',b'] = [a,\lambda]$. We recall that 
\[ g'(\lambda) = B' \left( 1  + \frac{4(b-\lambda)}{b-a}   \right) , \quad g(\lambda) = B - B' \left(1  + \frac{2(b - \lambda)}{b-a}    \right) (b-\lambda) \]
then
\begin{align*}
& \Phi\Big(a',b',g(a'),g(b'),g'(b') \Big) \\
= & \frac{56 g'(b')}{L_\star(\lambda-a)} + \frac{32(g(\lambda)-g(b))}{L_\star(\lambda-a)^2 }  \\
= & \frac{56 g'(\lambda)}{L_\star(\lambda-a)} + \frac{32(A-g(\lambda))}{L_\star(\lambda-a)^2 } \\
= & \frac{56 B' }{L_\star(\lambda-a)}  + \frac{56B'}{L_\star(\lambda-a)} \times \frac{4(b-\lambda)}{b-a}  + \frac{32(A-B)}{L_\star(\lambda-a)^2 } + \frac{32 B'}{L_\star(\lambda-a)^2 } \times \frac{2(b-\lambda)^2}{b-a}\\
= & \frac{56 B' }{L_\star(\lambda-a)}  + \frac{56B'}{L_\star(\lambda-a)} \times \frac{4(b-\lambda)}{b-a}   + \frac{32 B'}{L_\star(b-a) } \times \frac{2(b-\lambda)^2}{(\lambda-a)^2} + \frac{32(A-B)}{L_\star(\lambda-a)^2 }\\
\leq & \frac{56 B' }{L_\star(\lambda-a)}  + \frac{56B'}{L_\star(\lambda-a)} \times 2  +  \frac{32 B'}{L_\star(b-a) } \times \frac{1}{2} + \frac{32(A-B)}{L_\star(\lambda-a)^2 }
\end{align*}
with the last inequality using that $2(b-\lambda)\leq b-a$, then 
\begin{align*}
& \Phi\Big(a',b',g(a'),g(b'),g'(b') \Big) \\
\leq & \frac{5}{2} \times \frac{56 B' }{L_\star(\lambda-a)}  + \frac{32(A-B)}{L_\star(\lambda-a)^2 } \\
\leq & 5 \times \frac{56 B' }{L_\star(b-a)}  + 4 \times \frac{32(A-B)}{L_\star(b-a)^2 } = 5 \Phi(a,b,A,B,B')
\end{align*}
Combining both, we also guarantee that 
\[ \Phi( a',b',g(a'),g(b'),g'(b')) \leq 5 \Phi(a,b,A,B,B') \]
\end{proof}

In the previous theorem, we showed that the condition evaluation function $\Phi$ grows at an exponential speed at most through each evaluation and we know that $\Phi \leq 1$ means that we can always find two smooth functions which have distinct intervals that satisfy \eqref{eq:g.eps.opt}. Combining both, the next theorem gives an explicit formulation for the logarithmic complexity

\begin{theorem}
\label{induced-lower-bound-binary}
Let's fix parameters $C,\epsilon, L_\star > 0$. Consider a sequence of $N$ points $(\lambda_1,\ldots,\lambda_N)$ where $(g(\lambda_i),g'(\lambda_i))$ are evaluated. If the number of iterations is insufficient
\[ N < \log(5) \Big( \log \frac{L_\star}{\epsilon} + \log \frac{C}{(C+1)^2} - \log(88) \Big)  \]
Then there exists two $L_\star$-smooth functions $(g_1,g_2)$ and two sets $I_1, I_2$ such that
\begin{align*}
\forall  1\leq i \leq N, \quad g_1(\lambda_i) = g_2(\lambda_i) = g(\lambda_i) \\
\forall  1\leq i \leq N, \quad g_{1}'(\lambda_i) = g_{2}'(\lambda_i) = g'(\lambda_i) \\
\forall \lambda \in I_1, \quad  Cg_1(\lambda) + \lambda g_{1}'(\lambda) \geq \epsilon \\
\forall \lambda \in I_2, \quad  Cg_2(\lambda) + \lambda g_{2}'(\lambda) \geq \epsilon \\
I_1 \cup I_2 = [0,1]
\end{align*}
\end{theorem}

\begin{proof}
The lower depends on the initial value of $\Phi$. We can consider $[a,b]=[\frac{1}{C_\star},1]$ with $A=\frac{\epsilon}{C}$, $B = 0$ and $B'=C_\star \epsilon$.  We recall that $C_\star = \frac{C+1}{C}$ ($1-\frac{1}{C_\star} = \frac{1}{C+1}$) and we notice that
\begin{align*}
\Phi\left(\frac{1}{C_\star},1,\frac{\epsilon}{C},0,C_\star \epsilon\right) & = \frac{56 C_\star \epsilon (C+1) }{L_\star}  + \frac{32 \epsilon (C+1)^2}{C  L_\star } \\
& = \frac{56 \epsilon (C+1)^2 }{CL_\star}  + \frac{32 \epsilon (C+1)^2}{C  L_\star } = \frac{88 \epsilon (C+1)^2 }{CL_\star}
\end{align*}

Then the minimum iteration is 
\[ \log_5 \Big( \frac{CL_\star}{88 \epsilon (C+1)^2 }\Big) = \log(5) \Big( \log \frac{L_\star}{\epsilon} + \log \frac{C}{(C+1)^2} \Big) + \mathcal{O}(1)\]
\end{proof}

\section{Lower bounds in p-norm}

For the lower bound, we use the results from to~\cite{Diakonikolas2024} where the authors consider convex functions that are weakly smooth in a non-Euclidean norm. 

\begin{theorem}[\cite{Diakonikolas2024}]
Let $1 \leq p \leq \infty$, and consider the problem class of unconstrained minimization with objectives in the class $\mathcal{F}_{\mathbb{R}^d,\|\cdot\|_p}(\kappa, L)$, whose minima are attained in $\mathcal{B}_{\|\cdot\|_p}(0, R)$. Then, the complexity of achieving additive optimality gap $\epsilon$, for any local oracle, is bounded below by:
\[
\left\{
\begin{array}{lcl}
\Omega\left(\left(\frac{L R^\kappa}{\epsilon[\ln d]^{\kappa-1}}\right)^{\frac{2}{3 \kappa-2}}\right), & & \mbox{ if } 1 \leq p<2 ; \\
\Omega\left(\left(\frac{L R^\kappa}{\epsilon \min \{p, \ln d\}^{\kappa-1}}\right)^{\frac{p}{\kappa p+\kappa-p}}\right), & & \mbox{ if }  2 \leq p<\infty ; \text { and, } 
\end{array}
\right.
\]
The dimension $d$ for the lower bound to hold must be at least as large as the lower bound itself.
\end{theorem}

We would like to point out that the previous lower bounds use convex functions whereas the upper bounds provided in our work focus on the star-convex function class. Therefore, it is surprising that our algorithms performs nearly optimally on $p$-norms up to a factor depending only on $\tau$.

\section{Bregman divergence and radius of domain}\label{bregman vs radius}
In this section we relate the Bregman divergence generated by the $q$-th power of the $p$-norm to the $q$-th power of the $p$-norms of the arguments.

\begin{lemma}
  Fix $p>1$, $q > 1$ and define $\psi(x) := \frac{1}{q} \norm{x}_p^q$. Then
  \[
    D_\psi(x,y)
    ~\le~
    2 \max \big\{
    \norm{x}_p^q,
    \norm{y}_p^q
    \big\}
    .
  \]
\end{lemma}

\begin{proof}
By definition, we have
\[
  D_\psi(x,y)
  ~=~
  \frac{1}{q} \norm{x}_p^q
  - \frac{1}{q} \norm{y}_p^q
  - \norm{y}_p^{q-p} \sum_i (x_i-y_i) |y_i|^{p-2} y_i
  ,
\]
which we may reorganize to
\[
  ~\le~
  \frac{1}{q} \norm{x}_p^q
  + \frac{q-1}{q} \norm{y}_p^q
  +  \norm{y}_p^{q-p} \sum_i x_i |y_i|^{p-1}
  .
\]
Seeing that sum as an inner product, by H\"older with $1 = \frac{1}{p} + \frac{1}{\frac{p}{p-1}}$, this is bounded above by
\[
  ~\le~
  \frac{1}{q} \norm{x}_p^q
  + \frac{q-1}{q} \norm{y}_p^q
  + \norm{x}_p \norm{y}_p^{q-1}
\]
and further by convexity of the exponential (i.e.\ $a^\theta b^{1-\theta} = e^{\theta \ln a + (1-\theta) \ln b} \le \theta a + (1-\theta) b$) by
\[
  ~\le~
  2 \left(
    \frac{1}{q} \norm{x}_p^q
    + \frac{q-1}{q} \norm{y}_p^q
  \right)
  .
\]
The result follows by observing that a convex combination is bounded by the maximum.
\end{proof}

\section{Application for l1 norm}
\label{appendix-p=1}
For the $\ell_1$ norm, we can apply our algorithm for $\ell_q$ norm with $q=1+s$ and use the equivalence of norms :
\begin{align*}
D_F(x,y) \leq \frac{L}{\kappa}\|x-y \|_1^{\kappa} \leq \frac{Ld^{\kappa(1- \frac{1}{q})}}{\kappa }\|x-y \|_q^{\kappa} = \frac{Ld^{\frac{\kappa s}{s+1}}}{\kappa }\|x-y \|_q^{\kappa} 
\end{align*}
Since $F$ is $(Ld^{\frac{\kappa s}{s+1}},\kappa)$-weakly-smooth with respect to the $q$-norm, we can pick $\psi(x) = \frac{1}{2}\|x\|_q^2$ as the distance generating function; it is strongly convex with respect to the $q$-norm.

Now we have a weakly smooth function and a strongly convex distance generating function with respect to $q$-norm, our algorithms leads to the following precision \[  \mathcal{O}_{\kappa} \left( d^{\frac{2\kappa s}{(s+1)(3\kappa-2)}} \left(\frac{L \tau^{2} R^\kappa}{\epsilon}\right)^{\frac{2}{3\kappa - 2 }} \log^2 \Big( \frac{L \tau R}{\epsilon} \Big) \right)  \] with $s>0$ a parameter that can be chosen arbitrary small. Therefore, the dimension cost can be controlled asymptotically when $d\xrightarrow{} \infty, \epsilon \xrightarrow{} 0$.

\section{Smooth case where q=2}
\label{app:qk2}

For the completeness of the theory, in the special case where $q=\kappa=2$, we present a variation of our Theorem~\ref{thm-p>2} here. We start by recalling the definitions of the standard smoothness and strong convexity. 

\begin{definition}[Smoothness]
\label{def-smmoth}
A continuously differentiable function  $F :\mathbb{R}^d  \xrightarrow{} \mathbb{R} $ is said to be $L$-smooth with respect to norm $\|\cdot \|$ if for all $x,y \in \mathbb{R}^d$
\[ | D_F(x,y) | \leq \frac{L}{2} \|x-y\|^{2}.\]
\end{definition}
\begin{definition}[Strong convexity]
\label{def-stronglyx-convex}
A continuously differentiable function  $F :\mathbb{R}^d  \xrightarrow{} \mathbb{R} $ is said to be $\mu$-strongly convex with respect to norm $\|\cdot \|$ if for all $x,y \in \mathbb{R}^d$
\[ D_F(x,y) \geq \frac{\mu}{2} \|x-y\|^{2}.\] 
\end{definition}

\begin{theorem}
\label{thm-p=2}
In the setting of Assumption~\ref{ass},  Algorithm~\ref{algo-acc} with the tuning below gives for all $t \geq 1 $
\begin{align*}
   A_t \Big( F(x_{t+1}^{ag}) - F_\star \Big) \leq  &  D_{\psi}(x_\star ,x_t) -  D_{\psi}(x_\star,x_{t+1}) + \eta_t \epsilon_t  + A_{t-1} \Big( F(x_t^{ag}) - F_\star  \Big)     
\end{align*}
where $\mathcal{O}_{q, \kappa}$ omits constants depending only on $(q,\kappa)$. This is achieved for any $\alpha > 0$ by
\begin{equation}
\begin{aligned}
\alpha_t  :=  \alpha, \quad \quad   \eta_t := \frac{\alpha t}{2\tau} , \quad \quad  \epsilon_t :=  \frac{1}{ t \eta_t }, \quad \quad A_t := \frac{\eta_t^2}{\alpha_t} .
\end{aligned}
\end{equation}

\end{theorem}

\begin{proof}
Compared to the case where $q>\kappa$, the step is the same until equation \eqref{mirror_analysis}, then the smooth analysis changes.
\begin{equation}
\begin{aligned}\label{q=k=2-part1}
 & \frac{\eta_t}{\tau} \Big( F(x_t^{md}) - F_\star \Big)  \\
{}\leq{} &  D_{\psi}(x_\star ,x_t) -  D_{\psi}(x_\star,x_{t+1}) +  \eta_t \langle \nabla F(x_t^{md}),x_t^{md} - x_t  \rangle \\
 & {}+{}\eta_t \langle \nabla F(x_t^{md}),x_{t} - x_{t+1}  \rangle  - \frac{\mu}{2} \| x_{t+1} - x_{t} \|^2.
\end{aligned}
\end{equation}
For the smooth analysis part, we still obtain :
\begin{equation}
\begin{aligned}
\alpha_t \Big( F(x_t^{md}) - F(x_{t+1}^{ag}) \Big) & \geq \mu \| x_{t}^{md} - x_{t+1}^{ag} \|^2 - \alpha_t D_F(x_{t+1}^{ag}, x_t^{md}) \\
& \geq   \mu \| x_{t}^{md} - x_{t+1}^{ag} \|^2  - \frac{L\alpha_t}{2}\| x_{t}^{md} - x_{t+1}^{ag} \|^2, 
\end{aligned}
\end{equation}
We will assume $\frac{2\mu }{L} > \alpha_t $. Similarly to the case $q >\kappa$, the smoothness of $F$ combining with the choice of $x_{t+1}^ag$  leads to
%
\begin{equation}\label{q=k=2-part2}
\begin{aligned}
& \eta_t \langle  \nabla F(x_t^{md}) , x_{t} - x_{t+1} \rangle  - \frac{\mu}{2} \| x_{t+1} - x_{t} \|^2  \\
\leq & \frac{\mu  \eta_t^{2}}{\alpha_t^{2}} \frac{\| x_{t+1}^{ag} - x_t^{md} \|^{2} }{2}  \\
\leq & \frac{\mu \eta_t^{2}}{ 2 \alpha_t^{2}} \Bigg( \frac{2\alpha_t}{2\mu - L \alpha_t} \Big( F(x_t^{md}) - F(x_{t+1}^{ag}) \Big) \Bigg) \\
= & \frac{\mu \eta_t^{2}}{  \alpha_t(2\mu - L\alpha_t) }\Big( F(x_t^{md}) - F(x_{t+1}^{ag}) \Big).  \\ \\
\end{aligned}
\end{equation}
Combining with previous equations \eqref{q=k=2-part1} \eqref{q=k=2-part2}, we know that
\begin{equation}
\begin{aligned}
 \frac{\eta_t}{\tau} \Big( F(x_t^{md}) - F_\star \Big)  & \leq{}   D_{\psi}(x_\star ,x_t) -  D_{\psi}(x_\star,x_{t+1}) +  \eta_t \langle \nabla F(x_t^{md}),x_t^{md} - x_t  \rangle \\
 & {}+{}\frac{\mu \eta_t^{2}}{  \alpha_t(2\mu - L\alpha_t) }\Big( F(x_t^{md}) - F(x_{t+1}^{ag}) \Big).
\end{aligned}
\end{equation}
which can be arranged into :
\begin{equation}
\begin{aligned}
 &  \frac{\mu \eta_t^{2}}{  \alpha_t(2\mu - L\alpha_t) }  \Big( F(x_{t+1}^{ag}) - F_\star \Big) \\
\leq  &  D_{\psi}(x_\star ,x_t) -  D_{\psi}(x_\star,x_{t+1}) + \eta_t \langle \nabla F(x_t^{md}),x_t^{md} - x_{t}  \rangle \\
 & + \Big(   \frac{\mu \eta_t^{2}}{  \alpha_t(2\mu - L\alpha_t) }  -  \frac{\eta_t}{\tau} \Big) \Big( F(x_t^{md}) - F(x_t^{ag})  \Big)  \\
& + \Big(   \frac{\mu \eta_t^{2}}{  \alpha_t(2\mu - L\alpha_t) }  -  \frac{\eta_t}{\tau} \Big) \Big( F(x_t^{ag}) - F_\star  \Big)  
.
\end{aligned}
\end{equation}
with binary search 
\begin{align*}
 \eta_t \langle \nabla F(x_t^{md}),x_t^{md} - x_{t}  \rangle 
+   \frac{\mu \eta_t^{2}}{  \alpha_t(2\mu - L\alpha_t) }  \Big( F(x_t^{md}) - F(x_t^{ag})  \Big) \leq \eta_t \epsilon_t 
\end{align*}
We write $A_t := \frac{\eta_t^2}{\alpha_t} $, then
\begin{equation}
\begin{aligned}
   A_t \Big( F(x_{t+1}^{ag}) - F_\star \Big) \leq  &  D_{\psi}(x_\star ,x_t) -  D_{\psi}(x_\star,x_{t+1}) + \eta_t \epsilon_t  + A_{t-1} \Big( F(x_t^{ag}) - F_\star  \Big)  .
\end{aligned}
\end{equation}
\end{proof}

However, when it comes to the analysis of the upper bound of the iterates, an additional assumption needs to be considered. If the following condition is not satisfied, we cannot theoretically guarantee the polynomial growth of our iterates. 
\begin{assumption}
There exists $C\geq 1$, such that for all $x \in \mathbb{R}^d$,
\[ \| \nabla \psi(x)\|_\ast \leq C \| x  \| \]
\end{assumption}

We note that this assumption is verified for $p$-norms with $1<p<2$ where we consider $\psi(x) : = \frac{\| x\|_p^2}{2}$. There the assumption is equivalent to $\| \phi(x) \|_\ast \leq C\| x\|$ with $\phi(x) := \nabla \Big( \frac{\| x\|^2}{2}\Big)$. That condition is satisfied for $C=1$. Now, we can provide the upper bound for the bounded iterates.

\begin{theorem} [case $q=\kappa=2$]
\label{thm-qk2-bounded}
We suppose that :
\[ \forall x \in \mathbb{R}^d, \quad \quad \| \nabla \psi(x)\|_\ast \leq C \| x  \| \]
Let's assume that $D_{\psi}(x_\star,x_1) \leq B $. We suppose that we are running Algorithm~\ref{algo-acc} with exponents $n_1 \geq 0,n_2 \geq 1$, such that for all $s \geq 1$ 
\begin{align*}
\max \left( \sum_{t=1}^{s} \eta_t \epsilon_t, \frac{\alpha_t}{\eta_t},  L\eta_s , \frac{L\alpha_s}{2}    \right)\leq  K^{n_1}  s^{n_2}
\end{align*}
and $K^{n_1} > B$ that might also depend on $(C, L,\tau,\mu)$. Then, there exists an upper bound for $\| x_t - x_{t}^{ag}\|$ which is polynomial in $t$.
\end{theorem}

\begin{proof} We use the result from Theorem~\ref{thm-p>2}, by telescopic sum, we know that :
\begin{align*}
D_{\psi}(x_\star,x_{\tau+1}) + A_\tau \Big( F(x_{\tau+1}^{ag}) - F(x_\star )\Big) \leq D_{\psi}(x_\star,x_1) + A_1\Big(F(x_1) - F(x_\star) \Big) + \sum_{t=1}^{\tau} (B_t + \eta_t \epsilon_t)  
\end{align*}
We recall that from the smoothness of $F$:
\begin{align*}
\| \nabla F(x_t^{md}) \|_{\ast} & \leq \frac{L}{2}\|x_t^{md} - x_{\star} \|\\
F(x_1) - F(x_\star) \leq D^{F}(x_1,x_\star) & \leq \frac{L}{2}\|x_1 - x_\star \|^{2} \leq \frac{qL}{2} B 
\end{align*}
which implies that 
\begin{align*}
\mu \| x_{\tau+1} - x_\star \|^2  &\leq B +  L B + K^{n_1} \tau^{n_2} \\
 \| x_{\tau+1} - x_\star \| & \leq \Big( \frac{B}{\mu} +  \frac{LB}{\mu}  + \frac{K^{n_1}}{\mu} \tau^{n_2} \Big)^{\frac{1}{2}} \\
 & \leq \Big( \frac{B}{\mu} +  \frac{LB}{\mu} \Big)^{\frac{1}{2}} + \frac{K^{\frac{n_1}{2}}}{\sqrt{\mu}} \tau^{\frac{n_2}{2}} 
\end{align*}
Since our setting is unconstrained, the (sub-)gradients have to cancel :
\begin{align*}
\eta_t \nabla F(x_t^{md})+   \nabla \psi(x_{t+1}) - \nabla \psi(x_t) & = 0  \\
\alpha_t \nabla F(x_t^{md}) + \phi (x_{t+1}^{ag} - x_t^{md}) & = 0 
\end{align*}
with $\phi(x) := \nabla \Big( \frac{\|x \|^2}{2}\Big)$.  With some manipulation 
\begin{align*}
\eta_t \| \nabla F(x_t^{md}) \|_\ast & \leq \| \nabla \psi(x_{t+1})\|_\ast + \| \nabla \psi(x_{t})\|_\ast  \leq C \Big( \| x_{t+1} \| + \| x_t \| \Big) \\
 & \leq C \Big( \| x_{t+1} - x_\star \| + \| x_t - x_\star \| + 2 \| x_\star \|\Big) \\
\alpha_t \| \nabla F(x_t^{md}) \|_\ast & = \| \phi (x_{t+1}^{ag} - x_t^{md}) \|_\ast  = \| x_{t+1}^{ag} - x_t^{md}\|
\end{align*}
We use the definition of $\phi$, we combine the previous step with Cauchy–Schwarz inequality
\begin{align*}
\| x_{t+1}^{ag} - x_{t}^{md} \|^2 & \leq  \langle \phi(x_{t+1}^{ag} - x_{t}^{md}) , x_{t+1}^{ag} - x_{t}^{md} \rangle \\
& \leq \alpha_t   \| \nabla F(x_t^{md}) \|_{\ast} \times \|  x_{t+1}^{ag} - x_{t}^{md} \| \\
\| x_{t+1}^{ag} - x_{t}^{md} \|   & \leq \alpha_t   \| \nabla F(x_t^{md}) \|_{\ast}
\end{align*}
From the triangle inequality, we obtain :
\begin{align*}
\| x_{t+1}^{ag} - x_\star \| & \leq \| x_{t+1}^{ag} - x_t^{md}\| + \|x_t^{md} - x_\star\|  \\
& \leq \alpha_t   \| \nabla F(x_t^{md}) \|_{\ast} + \lambda_t  \|x_t^{ag} - x_\star\|  + (1-\lambda_t) \|x_t - x_\star\| \\
& \leq \frac{C\alpha_t}{\eta_t} \Big( \| x_{t+1} - x_\star \| + \| x_t - x_\star \| + 2 \| x_\star \|\Big)+ \|x_t - x_\star\| +   \|x_t^{ag} - x_\star\|  
\end{align*}
By induction,
\begin{align*}
\| x_{\tau+1}^{ag} - x_\star \| & \leq \sum_{t=0}^{\tau} \left( \frac{C\alpha_{t+1}}{\eta_{t+1}} + \frac{C\alpha_t}{\eta_t} + 1 \right) \| x_{t+1} - x_\star \| + C \| x_\star \|\sum_{t=1}^{\tau} \frac{\alpha_t}{\eta_t} \\
& \leq C K^{n_1}\sum_{t=0}^{\tau} \left( (t+1)^{n_2} +  t^{n_2} + 1 \right) \| x_{t+1} - x_\star \| + C  K^{n_1} \| x_\star \| \tau^{n_2 }\\
\end{align*}
As we recall that \[ \| x_{t+1} - x_\star\| \leq  \Big( \frac{B}{\mu} +  \frac{LB}{\mu} \Big)^{\frac{1}{2}} + \frac{K^{\frac{n_1}{2}}}{\sqrt{\mu}} t^{\frac{n_2}{2}}, \]
we know that $\| x_{t+1}- x_\star\|$ has an upper bound polynomial in $t$. Therefore, the distance $\| x_t - x_{t}^{ag}\|$ is polynomial in $t$ as well.
\end{proof}

Combining Theorem~\ref{thm-p=2} with Theorem~\ref{thm-qk2-bounded}, we obtain the final convergence rate :

\begin{corollary}[case $q=\kappa=2$]
We suppose that :
\[ \forall x \in \mathbb{R}^d, \quad \quad \| \nabla \psi(x)\|_\ast \leq C \| x  \| \]
  If a bound $\frac{1}{\mu} D_{\psi}( x_\star ,x_1) \le  B$ is available, then the tuning of Theorem~\ref{thm-p=2} with
  $\alpha
  =
  \frac{\mu B}{L}$
  guarantees
  $F(x_{T+1}^{ag}) - F_\star
    =
    \mathcal{O}  \left(
      \frac{
        L\tau^{2}B}{T^2 } 
    \right)
    $.
The algorithm's oracle usage over $T$ iterations is upper-bounded by $\mathcal{O}(T \log(LB\tau T))$, where this bound represents the maximum number of times the oracle is called during the entire execution process.
\end{corollary}
\end{document}